\documentclass[10pt,twoside, english]{article}
\usepackage{psfrag,babel}
\usepackage{latexsym}
\usepackage{amsmath,amsthm}
\usepackage{amsfonts}
\usepackage{amssymb}
\usepackage{graphicx}
\usepackage{color}
\usepackage{epsfig}

\usepackage[modulo,right]{lineno}

\usepackage{hyperref}
\hypersetup{colorlinks=true, citecolor=blue}

\def\leq{\leqslant}
\def\geq{\geqslant}
\def\N{\mathbb{N}}
\def\R{\mathbb{R}}

\DeclareMathOperator{\diam}{diam}
\DeclareMathOperator{\Int}{Int}

\newtheorem{Proposition}{Proposition}[section]

\newtheorem{lemma}[Proposition]{Lemma}
\newtheorem{definition}[Proposition]{Definition}
\newtheorem{Corollary}[Proposition]{Corollary}

\newtheorem{remark} [Proposition]{Remark}

\newtheorem{theorem}[Proposition]{Theorem}

\graphicspath{{Figures/}}

\begin{document}
%\linenumbers
\title{Complexity of injective piecewise contracting  interval maps}

\author{E.\ Catsigeras$^1$, P.\ Guiraud$^2$ and A.\ Meyroneinc$^3$}
%\date{}
\maketitle
\begin{center}
\begin{small}
$^1$ Instituto de Matem\'{a}tica y Estad\'istica Rafael Laguardia, Universidad de la Rep\'{u}blica,\\
Montevideo, Uruguay\\
\texttt{eleonora@fing.edu.uy}

$^2$ Centro de Investigaci\'on y Modelamiento de Fen\'omenos Aleatorios Valpara\'iso, Facultad de Ingenier\'ia, Universidad de Valpara\'{\i}so,\\
Valpara\'{\i}so, Chile\\
\texttt{pierre.guiraud@uv.cl}

$^3$ Departamento de Matem\'aticas, Instituto Venezolano de Investigaciones Cient\'ificas,\\
Apartado 20632, Caracas 1020A, Venezuela\\
\texttt{ameyrone@ivic.gob.ve}
\end{small}

\begin{abstract}

We study the complexity of the itineraries of injective piecewise contracting maps on the interval.
We prove that for any such map the complexity function of any itinerary is eventually affine.
We also prove  that the growth rate of the complexity is bounded  from above by  the number $N-1$ of discontinuities of the map. To show that this bound is optimal, we construct piecewise affine contracting maps whose itineraries all have the complexity $(N-1)n+1$. In these examples, the asymptotic dynamics takes place in a minimal Cantor set containing all the discontinuities.

\end{abstract}
\end{center}



\section{Introduction} \label{Introduction}

Let $(X,d)$ be a compact metric space and $\{X_i\}_{i=1}^N$ be a finite collection of $N \geq 2$ non-empty disjoint open subsets such that  $X = \bigcup_{i=1}^{N} \overline{X_i}$.
Let $f: X \to X$ and assume $f$ discontinuous on the set $\Delta := \{x\in \overline{X_i}\cap\overline{X_j}, i\neq j\in\{1,\dots,N\}\}$.
If there exits a constant $\lambda \in (0,1)$ such that for any $i \in \{1,\ldots,N\}$ the map $f$ satisfies
\begin{equation}\label{CONTRACTION}
d(f(x),f(y)) \leq \lambda \, d(x,y)\qquad \forall\,x,y\in X_i,
\end{equation}
then $f$ is called a {\em piecewise contracting map} and each element of the collection $\{X_i\}_{i=1}^N$ is called a {\em contraction piece}.

In  \cite{CGMU16}, we  explored the diversity of asymptotic dynamics of these systems, and proved that a rich dynamics can appear if the attractor contains discontinuity points.
In particular,  we exhibited three-dimensional  examples with exponential complexity and positive topological entropy. On the other hand, if the attractor does not contain discontinuity points, then its dynamics  is simple, just composed by a finite number of periodic orbits.

In the present paper, we remain interested in the diversity of the dynamics but we restrict the study to a class  of one-dimensional  piecewise contracting maps. Our objective is to
 determine  the range of all the possible complexity functions in the whole considered class.
In particular, we are interested in the relation between certain features of the discontinuity points and the complexity of the dynamics.

If a piecewise contracting map $f$ is defined on  a compact interval and  each contraction piece is an open interval, we say that  $f$ is a {\it piecewise contracting interval map}.
For these systems, it has been shown that generically the asymptotic dynamics is periodic, first for injective maps \cite{B06, N1, N2} and later for more general one-dimensional maps \cite{N3}.
In this paper, we are instead interested in the  non-periodic asymptotic dynamics.
In dimension one, there are   few known examples of piecewise contracting maps with   non-periodic attractors \cite{BC99,CFLM06,GT88,V87}, and  none of them has orbits that accumulate in more than one discontinuity point. Therefore, little is known about the maximum complexity of the dynamics when the interval map has an arbitrary (finite) number of discontinuity points.

By complexity of a map, we refer to the complexity function of the  itineraries of its orbits.
To define the itineraries of a piecewise contracting map $f$, consider the set $\widetilde X$  of those points of $X$ whose orbit never intersects $\Delta$, that is
\begin{equation}\label{XTILDE}
\widetilde X := \bigcap _{n= 0}^{+ \infty} f^{-n}(X\setminus\Delta),
\end{equation}
and assume that $\widetilde{X} $ is non-empty. We say that the sequence $\theta=\{\theta_t\}_{t\in\N} \in \{1,2, \dots, N\}^{\mathbb{N}}$ is the \emph{itinerary} of $x \in \widetilde{X}$ if for every $t \in \mathbb{N}$ and $i \in \{1,\ldots,N\}$ we have $\theta_t = i$  if and only if $f^t(x) \in X_i$. The \emph{complexity function} of
a sequence $\theta$ is the function defined for every $n\geq 1$  by
\[
p(\theta,n) := \# L_n(\theta)\qquad\text{where}\qquad L_n(\theta) := \{ \theta_t \dots \theta_{t+n-1}, \; t \geq 0 \} \qquad\forall\, n\geq 1,
\]
that is, $p(\theta, n)$  gives the number of different words of length $n$ contained in  $\theta$. Therefore, the complexity function of a sequence is  a non-decreasing function of $n$.
Also, if there exists $n_0\geq 1$ such that $p(\theta,n_0+1)=p(\theta,n_0)$, then $p(\theta,n)=p(\theta,n_0)$ for all $n\geq n_0$. This implies that a symbolic sequence is eventually periodic if and only if its complexity function is eventually constant.

In this paper we consider piecewise contracting maps satisfying a \lq\lq separation property\rq\rq. To define this property, first note, from inequality \eqref{CONTRACTION}, that for any $i\in\{1,\dots,N\}$ the restriction $f|_{X_i}$ of $f$ to the piece $X_i$ admits a continuous extension $f_i:\overline{X_i}\to X$  which also satisfies inequality \eqref{CONTRACTION} on $\overline{X_i}$.

\begin{definition} {\bf (Separation)} \label{definitionSeparationProperty}\em
We say that a piecewise contracting map $f$ satisfies the \em separation property \em if for every $i\in\{1,\dots,N\}$ the continuous extension $f_i:\overline X_i\to X$ is injective and $f_i(\overline X_i) \cap f_j(\overline X_j) = \emptyset$ for any $j\in\{1,\dots,N\}$ such that $j\neq i$.
\end{definition}

\noindent A map $f$ which satisfies the separation property is obviously injective on $X\setminus\Delta$, but not
necessarily on the whole set $X$. A map $f$ which is injective on $X$ does not satisfy the separation property if and only if there are $i$ and $j$ in $\{1,\dots,N\}$ such that
$\lim\limits_{x\to y}f|_{X_i}(x)=\lim\limits_{x\to z}f|_{X_j}(x)$ for some $y\in\Delta\cap\overline{X}_i$ and $z\in\Delta\cap\overline{X}_j$, with $y\neq z$ if $i=j$. It follows that not every injective map satisfies the separation property. Nevertheless, in dimension one, every injective map whose discontinuities are all of the first kind satisfies the separation property.

Our main result is the following Theorem \ref{TH}. We will later complement its statement with the additional results given by Theorem \ref{THCOMPNBLR} and Theorem \ref{TH3} about the relations between  the complexity function and the dynamical asymptotic behaviour of the orbits near  the discontinuity points.

\begin{theorem}\label{TH} 1) Let $\theta$ be the itinerary of an orbit of a piecewise contracting interval map which has $N$ contraction pieces and satisfies the separation property. Then, there exist $\alpha\in\{0,1,\dots, N-1\}$, $\beta\geq 1$ and $m_0\geq 1$, such that the complexity function of  $\theta$ satisfies
\begin{equation}\label{COMPTH}
p(\theta,n)= \alpha n+\beta \qquad\forall\,n\geq m_0,
\end{equation}
with $\beta=1$ if $\alpha=N-1$.

2) For every $N\geq 2$, there exists a piecewise affine contracting interval map $f$ which has $N$ contraction pieces and satisfies the separation property, such that
\begin{equation}\label{COMP2TH}
p(\theta,n)= (N-1) n+1 \qquad\forall\,n\geq 1,
\end{equation}
for every itinerary $\theta$ of $f$.
\end{theorem}

As mention above, generically in the space of piecewise contracting interval maps, all the orbits are attracted by periodic orbits. Therefore, generically, the itinerary of any orbit is eventually periodic and has an eventually constant complexity function. In other words, $\alpha=0$ in equality \eqref{COMPTH}. Nevertheless, non-periodic attractors do appear when some orbits accumulate on discontinuity points \cite{CGMU16}. In dimension two or three, this can produce itineraries of polynomial or exponential complexity \cite{CGMU16, KR06, LU06}. But in contrast,  Theorem \ref{TH} proves that in dimension one the complexity of any non-periodic itinerary is   affine.

To prove Theorem \ref{TH}, we will deduce equality \eqref{COMPTH} from precise results stated in Lemma \ref{PROCOMPNBNLR} and Theorem \ref{THCOMPNBLR}, which relate the complexity of an itinerary with the recurrence properties of the corresponding orbit arbitrarily near the discontinuity points.
In fact, the value of $\alpha$   equals the number of discontinuity points on which the orbit accumulates from both sides, and therefore is bounded above by the number of discontinuities contained in the attractor (which is at most $N-1$). On the other hand, unless $\alpha=(N-1)$, the constant $\beta$ depends on the transient behaviour of the dynamics and can be arbitrarily large; see relation \eqref{KBORNES}.

Part 1) of Theorem \ref{TH} states that the complexity of an itinerary is at most equal to $(N-1)n+1$.
For $N=2$, there are known examples of piecewise contracting maps whose itineraries have such a Sturmian complexity \cite{B93,BC99,C99,CFLM06,GT88,V87}. In those examples, the attractor is a Cantor set supporting a minimal dynamics.
Part 2) of Theorem \ref{TH} states that for any value of $N\geq 2$  there also exists  a piecewise contracting map
with ``full" complexity. Thereby, it establishes the optimality of the upper bound $N-1 $ for the growth coefficient $\alpha $ of the complexity function, for any number   $N \geq 2$ of contracting pieces.

To prove the existence of those maps with full complexity,  by induction on $N$ we construct for any $N\geq 3$ a piecewise contracting map whose complexity function satisfies equality \eqref{COMP2TH}, using as a base case a known example with $N=2$ and a Sturmian complexity.
As a consequence, we will prove with Theorem \ref{TH3} that the attractor of  each of  these maps inherit the Cantor structure and minimality of the attractor of the base case.

The proof of part 2) of Theorem \ref{TH}  provides for each $N \geq 2$ a map whose all itineraries have the same full complexity. But we note that the method also allows the construction of  examples for which different affine complexities coexist (for different orbits).

We prove part 1) of Theorem \ref{TH} in Section \ref{COMPI}, and part 2)  in Section \ref{SFULL}.

\section{Complexity of the itineraries}\label{COMPI}

\subsection{Preliminary general results on itineraries}

In this subsection we give some preliminary results that are not specific to piecewise contracting interval maps. In fact, here $X$ is not necessarily an interval and $f:X\to X$ may not satisfy the inequality \eqref{CONTRACTION}, provided it admits continuous extensions on each continuity piece $X_i$.

\begin{definition}\label{Def_Atoms}{\bf (Atoms)} \em
For every $i \in \{1, \dots, N\}$ let $F_i : \mathcal{P}(X) \to \mathcal{P}(X)$ be defined by $F_i(A) = \overline{f(A \cap X_i)}$ for all $A \in \mathcal{P}(X)$, where $\mathcal{P}(X)$ denotes the  set of parts of $X$.
Let $n \geq 1$ and $(i_1, \ldots, i_n) \in \{1, \dots, N\}^n$. We say that $A_{i_1, \ldots i_n}:=F_{i_n} \circ F_{i_{n-1}} \circ \cdots \circ F_{i_1}(X)$ is an \emph{atom of generation} $n$ if it is non-empty.
We denote $\mathcal{A}_n$ the set of all the atoms of generation $n$.
\end{definition}

\begin{remark}\label{RKATOMS}\em In the sequel we will often use the following basic properties of the atoms:
by construction,  $A_{i_1 i_2 \ldots i_n} \subset A_{i_2 \ldots i_n} \subset  \ldots \subset A_{i_n}$, and if $f$ is piecewise contracting then $\max\limits_{A\in\mathcal{A}_{n+1}}\diam(A)\leq\lambda\max\limits_{A\in\mathcal{A}_{n}}\diam(A)$ for all $n\geq 1$,
where $\diam(A)$ is the diameter of $A$.
\end{remark}

As shown by the following Lemma \ref{Atom_Unicity}, the separation property implies that the atoms of a same generation are pairwise disjoint.

\smallskip
\begin{lemma}\label{Atom_Unicity}
Suppose that $f$ satisfies the separation property. For every $n \geq 1$
\begin{enumerate}
 \item  if  $A, B \in \mathcal{A}_n$ are such that $A \cap B \neq \emptyset$, then $A=B$,
\item  if $A_{i_1 \ldots i_n}$, $A_{j_1 \ldots j_n}\in\mathcal{A}_n$ and $A_{i_1 \ldots i_n} = A_{j_1 \ldots j_n}$, then $(i_1, \ldots, i_n) = (j_1, \ldots, j_n)$.
\end{enumerate}
\end{lemma}

\begin{proof}
It is easy to show that 1) is true for $n=1$.
Suppose now that it is true for some $n>1$.
Let $A$,  $B \in \mathcal{A}_{n+1}$. Then there exists $C$ and $D \in \mathcal{A}_{n}$ and $i$, $j \in \{1,\ldots N\}$ such that $A = f_i(\overline{C \cap X_i})$ and $B = f_j(\overline{D \cap X_j})$.
Suppose that $A \cap B \neq \emptyset$.
Since $f_i(\overline{X}_i) \cap f_j(\overline{X}_j) = \emptyset$ for $i \neq j$, it follows that $i = j$.
Now, since $f_i$ is injective, if $A \cap B \neq \emptyset$ we have $(\overline{C \cap X_i}) \cap (\overline{D \cap X_i}) \neq \emptyset$, and $\overline{C} \cap \overline{D} \neq \emptyset$.
Since $\overline{C} \cap \overline{D} = C \cap D$, using the induction hypothesis we deduce that $C = D$ and it follows that $A =B$.

By the separation property 2) is true for $n=1$. Suppose it is true for some $n > 1$ and suppose $A_{i_1 \ldots i_{n+1}} = A_{j_1 \ldots j_{n+1}}$.
 Then,
 $f_{i_{n+1}}(\overline{A_{i_1 \ldots i_n} \cap X_{i_{n+1}}}) = f_{j_{n+1}}(\overline{A_{j_1 \ldots j_n} \cap X_{j_{n+1}}})$ and $j_{n+1} = i_{n+1}$, since $f_{i_{n+1}}(\overline{X}_{i_{n+1}}) \cap f_{j_{n+1}}(\overline{X}_{j_{n+1}}) \neq \emptyset$ implies that $j_{n+1} = i_{n+1}$.
 On the other hand, if $f_{i_{n+1}}(\overline{A_{i_1 \ldots i_n} \cap X_{i_{n+1}}}) = f_{i_{n+1}}(\overline{A_{j_1 \ldots j_n} \cap X_{i_{n+1}}})$ then by injectiveness $A_{i_1 \ldots i_n} \cap A_{j_1 \ldots j_n} \neq \emptyset$, which implies by 1) that $A_{i_1 \ldots i_n} = A_{j_1 \ldots j_n}$.
 Using the induction hypothesis it follows that $i_k = j_k$ for all $k \leq n$.
\end{proof}

The following Lemma \ref{Itinerary1} and Lemma \ref{Itinerary2} give the relation between the itinerary of a point of $\widetilde{X}$ and the atoms visited by the orbit of that point.

\begin{lemma}\label{Itinerary1}
Let $x \in \widetilde{X}$ and $\theta \in \{1,\dots,N\}^{\mathbb{N}}$ be its itinerary.
Then $f^{t+n}(x) \in A_{\theta_t \theta_{t+1} \ldots \theta_{t+n-1}}$ for every $t \geq 0$ and $n \geq 1$.
\end{lemma}

\begin{proof}
Let $t = 0$.
By the definitions of atom and itinerary we have that
$f(x) \in f(X_{\theta_0}) \subset A_{\theta_0}$ since $x \in X_{\theta_0}$.
Assume that $f^{n}(x) \in A_{\theta_0 \theta_1 \ldots \theta_{n-1}}$ for some $n >1$.
Then
$f^{n+1}(x) = f(f^{n}(x)) \in f(A_{\theta_0 \theta_1 \ldots \theta_{n-1}} \cap X_{\theta_{n}}) \subset A_{\theta_0 \theta_1 \ldots \theta_{n}}$.
Now suppose $t \neq 0$, let $y = f^t(x)$ and $\omega$ be the itinerary of $y$.
Then $f^{t+n}(x) = f^n(y)  \in A_{\omega_0  \ldots \omega_{n-1}} = A_{\theta_t \theta_{t+1} \ldots \theta_{t+n-1}}$.
\end{proof}

\begin{lemma}\label{Itinerary2}
Suppose that $f$ satisfies the separation property.
 Let $x \in \widetilde{X}$, $t \geq 0$, $n \geq 1$ and $\theta$ be the itinerary of $x$.
If $f^{t+n}(x) \in A_{i_1 i_2 \ldots i_n}$ then $ \theta_t \theta_{t+1} \dots \theta_{t+n-1} = i_1i_2 \ldots i_n$.
\end{lemma}

\begin{proof}
 Suppose $t = 0$.
 By Lemma \ref{Itinerary1} we have $f^n(x) \in A_{\theta_0 \theta_1 \ldots \theta_{n-1}}$, therefore $A_{i_1 \ldots i_n} \cap A_{\theta_0 \ldots \theta_{n-1}} \neq \emptyset$.
 By Lemma \ref{Atom_Unicity} we have $A_{i_1 \ldots i_n} = A_{\theta_0 \ldots \theta_{n-1}}$ and $\theta_0 \ldots \theta_{n-1} = i_1 \ldots i_n$.
 Now suppose $t \neq 0$, let $y = f^t(x)$ and $\omega$ be the itinerary of $y$.
Then $f^{t+n}(x) = f^n(y)  \in A_{i_1  \ldots i_{n}}$, which implies that $\omega_0\ldots\omega_{n-1}= i_1\ldots i_{n}$, that is $\theta_t \ldots \theta_{t+n-1}=i_1  \ldots i_{n}$.
\end{proof}

\begin{Corollary}\label{Itineraries_Atoms}
Suppose that $f$ satisfies the separation property.
 Let $x \in \widetilde{X}$, $t \geq 0$, $n \geq 1$, $\theta$ be the itinerary of $x$, and $(i_1, i_2, \ldots i_n) \in \{1 \ldots N\}^n$.
 Then
$
  \theta_t \theta_{t+1} \ldots \theta_{t+n-1} = i_1 i_2 \ldots i_n \; \text{ if and only if } \; f^{t+n}(x) \in A_{i_1 \ldots i_n}.
 $
\end{Corollary}

\begin{proof}
 It follows directly from Lemmas \ref{Itinerary1} and \ref{Itinerary2}.
\end{proof}

Let $x \in \widetilde{X}$, $I:=\{1,\dots,N\}$ and $\theta \in I^{\mathbb{N}}$ be the itinerary of $x$. Now, for any $n \geq 1$ and $k \in I$ consider the set
\begin{equation*}
 {L}_n^k(\theta) := \{ i_1 \dots i_n \in L_n(\theta) : \# \{j \in I : \exists t \geq 0 \text{ such that } f^{t+n}(x) \in A_{i_1 \dots i_n} \cap X_j \} = k \}.
\end{equation*}
A word of length $n$ of $\theta$ belongs to  ${L}_n^k(\theta)$ if it is the label of an atom that intersects  at least $k$ continuity pieces and  if the orbit of $f^n(x)$ visits exactly $k$ of these intersections. The following Lemma \ref{complexity_bound2.2} puts in  relation the growth of the complexity function of an itinerary $\theta$ and the cardinality of the sets ${L}_n^k(\theta)$.

\begin{lemma}\label{complexity_bound2.2}
Let  $x \in \widetilde{X}$ and $\theta \in I^{\N}$ be the itinerary of $x$.  Then
\begin{equation}\label{complexity_bound}
 p(\theta,n+1)  \leq  p(\theta,n) + \sum_{k=2}^{N} (k-1) \# {L}_n^k(\theta)\qquad\forall\,n\geq 1.
\end{equation}
If moreover $f$ satisfies the separation property, then (\ref{complexity_bound}) is an equality.
\end{lemma}

\begin{proof} Let $n\geq 1$, and observe that $L_n(\theta) = \bigcup_{k=1}^{N} {L}_n^k(\theta)$. First, we show the inclusion

\begin{equation}\label{INCLUSIONINTESTINAL}
 L_{n+1}(\theta) \subset
  \bigcup_{k=1}^N B^k_n(\theta),
\end{equation}
where
\[
B^k_n(\theta):=\bigcup_{i_1 \dots i_n \in {L}_n^k(\theta)} \left\{ i_1 \dots i_n i_{n+1} \in I^{n+1} : \exists t\geq 0 :  f^{t+n}(x)\in A_{i_1 \dots i_n} \cap X_{i_{n+1}}\right\}.
\]

Let $i_1 \dots i_{n+1} \in L_{n+1}(\theta)$. Then, there exists $t \geq 0$ such that $i_1 \dots i_{n+1} = \theta_t \dots \theta_{t+n}$, which implies $i_1 \dots i_n \in L^k_n(\theta)$ for some $k\in I$ (since
$i_1 \dots i_n \in L_n(\theta)$ and $f^{t+n}(x)\in X_{i_{n+1}}$). On the other hand, by Lemma \ref{Itinerary1}, we have $f^{t+n}(x) \in A_{i_1 \dots i_{n}}$. It follows that $i_1 \dots i_{n+1}\in B_n^k(\theta)$, and thus \eqref{INCLUSIONINTESTINAL} is true.

If we suppose moreover that $f$ satisfies the separation property, we can deduce that \eqref{INCLUSIONINTESTINAL} is an equality. Indeed, if $i_1 \dots i_{n+1}\in\bigcup_{k=1}^N B^k_n(\theta)$, then there exist $k\in I$ and $t\geq 0$ such that $i_1 \dots i_{n}\in L^k_n(\theta)$ and $f^{t+n}(x)\in A_{i_1 \dots i_{n}}\cap X_{i_{n+1}}$. The latter implies that $\theta_{t+n}=i_{n+1}$, and $i_1 \dots i_{n}=\theta_{t}\dots\theta_{t+n-1}$ by Lemma \ref{Itinerary2}. It follows that $i_1 \dots i_{n+1}\in L_{n+1}(\theta)$.

To finish the proof observe that for any $k\in I$ the set $B_n^k(\theta)$ is defined by the union of disjoint sets that satisfy
\[
\#\{ i_1 \dots i_n i_{n+1} \in I^{n+1} : \exists t\geq 0 :  f^{t+n}(x)\in A_{i_1 \dots i_n} \cap X_{i_{n+1}}\}=k\qquad\forall\, i_1 \dots i_n \in {L}_n^k(\theta),
\]
by definition of ${L}_n^k(\theta)$. So we have $\#B_n^k(\theta)=k\#L_n^k(\theta)$. Moreover, since $L_n^k(\theta)\cap L_n^{k'}(\theta)=\emptyset$ if $k\neq k'$, on the one hand  $B_n^k(\theta)\cap B_n^{k'}(\theta)=\emptyset$ if $k\neq k'$,  and on the other hand $p(\theta,n) = \sum_{k=1}^{N} \# {L}_n^k(\theta)$. We deduce that
\[
  \#\bigcup_{k=1}^N B^k_n(\theta)=\sum_{k=1}^N k \, \#{L}_n^k(\theta) = \sum_{k=1}^N \#{L}_n^k(\theta) + \sum_{k=1}^N (k-1) \, \# {L}_n^k(\theta) =  p(\theta,n) + \sum_{k=2}^N (k-1) \, \# {L}_n^k(\theta).
\]
Now, from \eqref{INCLUSIONINTESTINAL} we conclude that \eqref{complexity_bound} is true, and is an equality if $f$ satisfies the separation property.
\end{proof}

\subsection{Discontinuities of piecewise contracting interval maps and complexity}

From now on, we assume that the  phase space $X$ of $f$ is a compact interval of $\R$ and that the contraction pieces are open intervals in $X$. This implies, in particular that the atoms are closed intervals. Also, since the number of pieces is finite, the map $f$ has a finite number of discontinuities. We  label the  $N$ contraction pieces $\{X_i\}_{i=1}^N$ of $f$, in such a way that $X_1<X_2<\dots<X_N$.

\begin{definition} \em
Let $x\in X$ and for any $n\geq 1$ denote $\mathcal{A}_n(x):=\{A\in\mathcal{A}_n :\exists t\in\N : f^{t+n}(x)\in A\}$.
Let $c\in\Delta$ and $i\in\{1,\dots,N-1\}$ be such that $c=\overline{X}_i\cap\overline{X}_{i+1}$. Let $n \geq 1$, we say that $c$ is {\em n-left-right visited} (in short {\em nlr-visited})  by $\{f^k(x)\}_{k\in\N}$  if  there exists $A_n\in\mathcal{A}_n(x)$ such that $c\in A$ and
\[
 \{t \in\N : f^{t+n}(x) \in A_n\cap X_i \} \neq \emptyset  \quad \text{ and } \quad   \{t \in\N : f^{t+n}(x) \in
 A_n\cap X_{i+1} \} \neq \emptyset.
\]
We denote $\Delta^n_{lr}(x)$ the set of the discontinuities that are $n$lr-visited. We say that $c$ is {\em left-right recurrenly visited} (in short {\em lr-recurrently visited}) by $\{f^k(x)\}_{k\in\N}$ if $c\in\Delta^n_{lr}(x)$ for all $n \geq 1$. We denote $\Delta_{lr}(x)$ the set of the discontinuities that are lr-recurrently visited.
\end{definition}

Note that for any $x\in X$ and $n\geq 1$ we have $\Delta_{lr}(x)\subset\Delta^{n+1}_{lr}(x)\subset\Delta^n_{lr}(x)\subset\Delta$, because any atom of generation $n+1$ is contained in an atom of generation $n$.
Also, if $c\in\Delta_{lr}(x)$, then $c$ is an accumulation point (by the left and by the right) of the orbit of $x$, since the diameter of the atoms of a piecewise contracting map goes to $0$ as their generation goes to infinity (see Remark \ref{RKATOMS}).

\begin{lemma}\label{PROCOMPNBNLR} Let $f$ be a piecewise contracting interval map satisfying the separation property. Let $x\in\widetilde{X}$ and $\theta$ be its itinerary. Then,
\begin{equation}\label{COMPNP1}
\#\Delta^n_{lr}(x)\leq p(\theta,n+1)-p(\theta,n)\leq \#\Delta\qquad\forall\, n\geq 1.
\end{equation}
Moreover, if $n_0\geq 1$ is the smallest integer such that $\#A\cap\Delta\leq 1$ for any $A\in\mathcal{A}_n(x)$ with $n\geq n_0$, then
\begin{equation}\label{COMPNBNLR}
p(\theta,n+1)=p(\theta,n)+\#\Delta^n_{lr}(x) \qquad\forall\,n\geq n_0.
\end{equation}
\end{lemma}
Note that $n_0$ exists and is bounded above by the smallest $n\geq 1$ such that
\[
\max\limits_{A \in \mathcal{A}_n}\diam(A) < \min\limits_{i\in \{1,\dots,N\}}\diam(X_i),
\]
which in turn can be bounded above by a function of $\lambda$ and  the diameters of the continuity pieces (see Remark \ref{RKATOMS}).

\begin{proof}
Let $n\geq 1$. Suppose that $c\in\Delta_{lr}^n(x)$, then there exists $A \in \mathcal{A}_n(x)$ such that $c\in A$.
Moreover there exist $i\in\{1,\dots,N\}$ and $t\in \mathbb{N}$ such that $f^{t + n}(x) \in A\cap X_i$. Therefore, according to Lemma \ref{Itinerary1}, we have $A \cap A_{\theta_{t} \dots \theta_{t + n-1}}\neq\emptyset$, and after Lemma \ref{Atom_Unicity} we have that $A= A_{\theta_{t} \dots \theta_{t + n-1}}$. As there exists also $t'\in \mathbb{N}$ such that $f^{t' + n}(x) \in A\cap X_{i+1}$, we have that  $\theta_{t} \dots \theta_{t + n-1} \in L_n^k(\theta)$ for some $k\geq 2$. We deduce that
\[
\Delta_{lr}^n(x)\subset\bigcup_{k=2}^N\bigcup_{i_1\dots i_n\in L_n^k(\theta)}\left(A_{i_1\dots i_n}\cap \Delta_{lr}^n(x)\right)
\]
and it follows that
\[
\#\Delta_{lr}^n(x)\leq\sum_{k=2}^N\sum_{i_1\dots i_n\in L_n^k(\theta)}\#\left(A_{i_1\dots i_n}\cap \Delta_{lr}^n(x)\right).
\]
Now, if $A\in\mathcal{A}_n(x)$ and $\#A\cap\Delta^n_{lr}(x)=q$, then $A$ intersects at least $q+1$ continuity pieces that are visited by the orbit of $f^n(x)$. It follows that for any $k\geq 2$ and  $i_1\dots i_n\in L_n^k(\theta)$ we have that $\#\left(A_{i_1\dots i_n}\cap \Delta_{lr}^n(x)\right)\leq k-1$. We deduce that
\begin{equation}\label{BINF}
\#\Delta_{lr}^n(x)\leq\sum_{k=2}^N(k-1)\# L_n^k(\theta)\qquad\forall\, n\geq 1.
\end{equation}

Now, let $n\geq 1$ and $k\geq 2$. If $i_1\dots i_n\in L^k_n(\theta)$, then $A_{i_1\dots i_n}\neq \emptyset$ and
$A_{i_1\dots i_n}$ intersects at least $k$ continuity pieces. As $A_{i_1\dots i_n}$ is a closed interval and the continuity pieces are open intervals, it follows that $A_{i_1\dots i_n}$ contains at least $k-1$ elements of $\Delta$. Now, according to Lemma \ref{Atom_Unicity}, if  $i_1\dots i_n$ and $i'_1\dots i'_n$ are two different words of $L_n(\theta)$ then $A_{i_1\dots i_n}\cap A_{i'_1\dots i'_n}=\emptyset$.
It follows that
\begin{equation}\label{BSUP}
\#\Delta\geq\sum_{k=2}^N(k-1)\# L_n^k(\theta)\qquad\forall\, n\geq 1.
\end{equation}
Then, inequalities \eqref{COMPNP1} follow from \eqref{BINF}, \eqref{BSUP} and Lemma \ref{complexity_bound2.2}.

Let $n\geq n_0$. Then, for any $i_1\dots i_n\in L_n(\theta)$ the atom $A_{i_1\dots i_n}$ intersects at most two continuity pieces, and therefore $L_n^k(\theta)=\emptyset$ for all $k\geq 3$. Moreover,
for any $i_1\dots i_n\in L_n^2(\theta)$ the discontinuity contained in $A_{i_1\dots i_n}$ belongs to
$\Delta_{lr}^n(x)$. We deduce that
\begin{equation}\label{BSUPN0}
\# \Delta_{lr}^n(x)\geq \# L_2^k(\theta)=\sum_{k=2}^N(k-1)\# L_n^k(\theta), \qquad\forall\, n\geq n_0,
\end{equation}
which together with \eqref{BINF} and Lemma \ref{complexity_bound2.2} implies \eqref{COMPNBNLR}.
\end{proof}

\begin{theorem}\label{THCOMPNBLR} Let $f$ be a piecewise contracting map satisfying the separation property. Let $x\in\widetilde{X}$ and $\theta$ be its itinerary, then there exits $m_0\geq 1$ such that
\begin{equation}\label{COMPNBLR}
p(\theta,n)= n\#\Delta_{lr}(x)+ \beta(x) \qquad\forall\,n\geq m_0,
\end{equation}
with
\begin{equation}\label{KBORNES}
p(\theta,1)-\#\Delta_{lr}(x)\leq \beta(x)\leq p(\theta,1)-\#\Delta+ m_0(\#\Delta-\#\Delta_{lr}(x)).
\end{equation}
\end{theorem}

\begin{proof}
For any $c\in\Delta$, either $c\in\Delta_{lr}(x)$ or there exists $\nu(c ):=\min\{n\geq 1: c\notin\Delta_{lr}^n(x)\}$. As $\Delta^{n+1}_{lr}(x)\subset\Delta^n_{lr}(x)$ for all $n\geq 1$, it follows that for any $c\in\Delta\setminus\Delta_{lr}(x)$ we have that $c\notin\Delta_{lr}^n(x)$ for all $n\geq \nu( c)$. Therefore, if $n_1=\max\{\nu(c ), c\in\Delta\setminus\Delta_{lr}(x)\}$ if $\Delta\neq\Delta_{lr}(x)$, and $n_1=1$ otherwise, then $\Delta_{lr}^n(x)=\Delta_{lr}(x)$ for all $n\geq n_1$.

Let $m_0:=\max\{n_0,n_1\}$. Then we can write \eqref{COMPNBNLR} as
\[
p(\theta,n+1)-p(\theta,n)=\#\Delta_{lr}(x) \qquad\forall\,n\geq m_0,
\]
which implies $p(\theta,n)= p(\theta,m_0) + (n-m_0)\#\Delta_{lr}(x)$ for all $n\geq m_0$. It follows that \eqref{COMPNBLR} is true with
\begin{equation}\label{KM0}
\beta(x)=p(\theta,m_0) -m_0\#\Delta_{lr}(x).
\end{equation}
Recalling that $\#\Delta^n_{lr}(x)\geq\#\Delta_{lr}(x)$ for all $n\geq 1$, from \eqref{COMPNP1} we obtain that
\[
p(\theta,1)-\#\Delta_{lr}(x)\leq p(\theta,n)-n\#\Delta_{lr}(x)\leq p(\theta,1)-\#\Delta+n(\#\Delta-\#\Delta_{lr}(x))\qquad\forall\, n\geq 1,
\]
and setting $n=m_0$, we obtain \eqref{KBORNES} from \eqref{KM0}.
\end{proof}

 \begin{proof}[Proof of part 1) of Theorem \ref{TH}.]
For any $x\in\widetilde{X}$ with itinerary $\theta$ we have
\begin{equation}\label{PTHETA1}
p(\theta,1)-\#\Delta\leq 1\leq p(\theta,1)-\#\Delta_{lr}(x),
\end{equation}
which implies, in particular, that $1\leq \beta(x)$. Together with Theorem \ref{THCOMPNBLR}, this proves equality \eqref{COMPTH} of Theorem \ref{TH}. In fact, equality \eqref{COMPTH} follows from equality \eqref{COMPNBLR} with $\alpha=\#\Delta_{lr}(x)\in\{0,1,\dots,N-1\}$. Also, if $\alpha=N-1$, then $\#\Delta_{lr}(x)=\#\Delta$. So, from \eqref{PTHETA1} and \eqref{COMPNBLR}, we conclude that $\beta(x)=1$.
 \end{proof}

\begin{remark}\em Now we give some direct consequences of Theorem \ref{THCOMPNBLR} and we comment their relations with other results.  From Theorem \ref{THCOMPNBLR} it follows that:

\noindent{\bf 1. } There exists an itinerary with a  complexity function which  is not eventually constant if and only if there exists a discontinuity point $c $ which is lr-recurrently visited by an orbit of $\widetilde{X}$.  In particular, if the limit set of $f$ does not contain any discontinuity point, then the complexity function of every itinerary is eventually constant (recall that if $c\in\Delta_{lr}(x)$, then $c$ belongs to the $\omega$-limit set of $x$).  In this case, for any $x\in\widetilde{X}$ with itinerary $\theta$ we have
\[
p(\theta,1)\leq p(\theta,n)\leq p(\theta,1) + (m_0-1)\#\Delta\qquad\forall\, n\geq 1,
\]
and $p(\theta,n)=\beta(x)$ is constant for any $n\geq m_0$. Moreover, when the limit set of $f$ does not contain discontinuity points, there exists a smallest integer $m\geq 1$ such that  no atom of generation $m$ contains   discontinuities. This integer $m$ is an upper bound for $m_0$, which  provides a uniform upper bound on $\beta(x)$ through inequalities \eqref{KBORNES}.

\noindent{\bf 2. } If $\#\Delta_{lr}(x) = 1$ but $\Delta_{lr}(x) \neq \Delta$, then the itinerary $\theta$ of the orbit of $x$ satisfies $p(\theta,n)=n+\beta(x)$ for all $n$ large enough, where $\beta(x)$
may be larger than 1.
An example of a piecewise contracting map whose itineraries have such a complexity can be found in \cite{CFLM06}.
In  \cite{D99}, it is shown that, up to a prefix of finite length, a sequence of complexity $n+\beta$ is the image by a morphism of a Sturmian sequence. We conclude that, if $\#\Delta_{lr}(x) = 1$ but $\Delta_{lr}(x) \neq \Delta$, the itinerary of any orbit, is Sturmian up to a morphism. Hence, up to a morphism, it is the itinerary of an irrational rotation, with respect to a suitable partition of the circle.

\noindent{\bf 3. } If $\#\Delta_{lr}(x) > 1$, the itinerary $\theta$ may be that of an irrational rotation: in fact, for  some  adequate values of $\alpha $  and $\beta$, the itineraries of an irrational rotation with respect to a suitable partition of the circle may have a complexity function of the form $\alpha n + \beta$ for all $n$ large enough. However, not every sequence with such a complexity  is itself an itinerary of an  irrational rotation \cite{AB98}.

\noindent{\bf 4. } If all the discontinuities are lr-recurrently visited by the orbit of a point $x\in\widetilde{X}$, i.e. $\Delta_{lr}(x)=\Delta$, then $m_0=n_0$ (see the definition of $m_0$ in the proof of Theorem \ref{THCOMPNBLR}). Besides, from part 1 of Theorem \ref{TH}, we know that $\beta(x)=1$ in this case. So, equality \eqref{COMPNBLR} becomes
\begin{equation}\label{COMPMAX}
p(\theta,n)=(N-1)n+1\qquad\forall n\geq n_0.
\end{equation}
In the particular case where  the map has  two contraction pieces ($N=2$)  and the (unique) discontinuity point  is  lr-recurrently visited by the orbit of $x$, then $n_0=1$ and $\theta(x)$ is a Sturmian sequence.  Therefore, it is  also  an itinerary  of an irrational rotation. In general, if the itinerary $\theta (x)$ satisfies \eqref{COMPMAX} for some $N\geq 2$, then it has the complexity of an itinerary of a $N$-interval exchange transformation satisfying the so-called Keane's infinite distinct orbit condition \cite{F13,FZ08}.
In fact  it is proved in \cite{P16} the following result:   if a piecewise contracting  map $f$ has no periodic orbit,
and is such that the image of any discontinuity and each lateral limit of $f$  belong to $\widetilde{X}$,
then it is semi-conjugate to an interval exchange transformation. It is however not easy to exhibit examples satisfying these hypotheses, since generically a piecewise contracting interval map has periodic points.
In the next section, we will construct such examples for every $N\geq 2$, where furthermore equality  \eqref{COMPMAX} holds for any itinerary $\theta$.
\end{remark}

\section{Existence of piecewise contracting interval maps of full complexity}\label{SFULL}

In the previous section we proved Theorem \ref{THCOMPNBLR}, which implies that the complexity of the itinerary of any orbit of a piecewise contracting interval map satisfying the separation property is bounded from above by an affine function whose slope is equal to the number of discontinuities of the map. However, as far as we know, there is still no example of piecewise contracting interval maps with more than one lr-recurrently visited discontinuity.
The purpose of this section is to construct such examples. Even more, we will construct examples for which all the discontinuities are lr-recurrently visited
by all the orbits.
These maps generate itineraries with the maximal  complexity  for the fixed number $N$ of contracting pieces. We say that they have ``full" complexity.

We are also interested in the asymptotic dynamics of such  examples. It takes place in what we call the {\it attractor} $\Lambda$ of the piecewise contracting map $f:X\to X$. To define the attractor we first recall the  definition of the atoms $A \in {\mathcal A}_n$ of generation $n$   for any natural number $n \geq 1$ (see Definition \ref{Def_Atoms}). We define the attractor $\Lambda \subset X$ as follows:
\begin{equation}\label{ATTRACTOR}
\Lambda:=\bigcap_{n=1}^\infty\Lambda_n\quad\text{where}\quad\Lambda_n:=\bigcup_{A\in\mathcal{A}_n}A
\quad\forall\, n\geq 1.
\end{equation}
The sets $\Lambda_n$ can equivalently  be   recursively defined by $\Lambda_1:=\overline{f(X\setminus\Delta)}$
and $\Lambda_{n+1}:=\overline{f(\Lambda_{n}\setminus\Delta)}$ for all $n\geq 1$.

Note that the attractor $\Lambda$ is  nonempty
and compact. Besides, $\Lambda$ contains all the non-wandering and $\omega$-limit points; see \cite{CGMU16} for more details
and examples.

Precisely, in this section we prove the following theorem:

\begin{theorem}\label{TH3}  For every $N\geq 2$, there exists a piecewise affine contracting map which has $N$ contraction pieces and satisfies the separation property, whose  attractor is a minimal Cantor set, and
such that each of its discontinuities is lr-recurrently visited by any orbit.
\end{theorem}

\noindent Theorem \ref{TH3}, together with \eqref{COMPMAX}, proves immediately equation \eqref{COMP2TH} of Theorem \ref{TH} for any $n\geq n_0$. Later, we will prove that it is always possible to construct the maps in such a way that  $n_0= 1$ (see Lemma \ref{N01}), to end the proof of part 2) of Theorem \ref{TH}.

 Observe that  the attractor of the piecewise contracting map  of Theorem \ref{TH3} contains {\it all} the discontinuities of the map. In fact any lr-recurrently visited discontinuity belongs to the $\omega$-limit set of some orbit, and the attractor contains all the $\omega$-limit sets.

To prove Theorem \ref{TH3}, we will prove  the following stronger statement:

\noindent {\bf Assertion A:} For every $N\geq 2$, there exists $N$ ordered disjoint open intervals $X_1=[c_0,c_1),X_{2}=(c_1,c_2),\dots,X_{N}=(c_{N-1},c_{N}]$ of  $X:=[c_0,c_{N}]$ and $f:X\to X$ with all the following properties:

\noindent {\bf P1)} The map $f$ is  piecewise contracting with contraction pieces $X_1,\dots,
X_{N}$ and $f|_{X_{i}}$ is affine with slope $\lambda\in(0,1)$.

\noindent {\bf P2)} The map $f$ satisfies the separation property.

\noindent {\bf P3)} The attractor $\Lambda$ of $f$ is a Cantor set.

\noindent {\bf P4)} The set $\bigcup_{i=1}^{N-1}\{f_{i}(c_i),f_{i+1}(c_i)\}$ is a subset of $\widetilde{X}$.

\noindent {\bf P5)} There exists $i\in\{1,\dots,{N-1}\}$ such that $\{f^n(f_{i}(c_i))\}_{n\in\N}$ or $\{f^n(f_{i+1}(c_i))\}_{n\in\N}$ is dense in $\Lambda$.

\noindent {\bf P6)} For any $x\in\widetilde{X}$ and $i\in\{1,\dots,{N-1}\}$ we have that $c_i\in\Delta_{lr}(x)$.

\medskip

To prove Assertion A and Theorem \ref{TH3}, we will need the following lemma.

\begin{lemma}\label{DENSITY} Let ${N}\geq 2$ and $c_0<c_{1}<\dots<c_{N}$ in $\R$. Let $f:[c_0,c_{N}]\to[c_0,c_{N}]$ be a piecewise contracting map  with contraction pieces $X_1=[c_{0},c_1), X_2=(c_{1},c_2),\dots, X_{N}=(c_{{N-1}},c_{N}]$ and which satisfies the separation property.  Suppose that there exists $i\in\{1,\dots,{N-1}\}$
such that

\noindent 1)  $f_j(c_i)\in\widetilde{X}$ and $\{f^n(f_{j}(c_i))\}_{n\in\N}$ is dense in $\Lambda$ for some $j\in\{i,i+1\}$.

\noindent 2) $c_i\in\Delta_{lr}(x_0)$ for some $x_0\in\widetilde{X}$.

\noindent Then, for any $\epsilon>0$ and $y\in\Lambda$ such that $\Lambda\cap (y,y+\nu)\neq\emptyset$  (resp. $\Lambda\cap (y-\nu,y)\neq\emptyset$) for all $\nu>0$,  there exists $l \geq 0$ such that $f^l(x_0)\in(y,y+\epsilon)$ (resp. $f^l(x_0)\in(y-\epsilon,y)$).
\end{lemma}

\begin{proof}  We will make the proof for $y\in\Lambda$ such that $\Lambda\cap (y,y+\nu)\neq\emptyset$ for all $\nu>0$, and without loss of generality we will suppose that $i=j=1$. Let $\epsilon>0$,  $z\in \Lambda_f\cap (y,y+\epsilon)$ and $\delta=\frac{1}{2}\min\{z-y,y+\epsilon-z\}$. Since $\{f^n(f_1(c_1))\}_{n\in\N}$ is dense in $\Lambda_f$, there exists $n$ such that $|f^n(f_1(c_1))-z|<\delta$. By injectivity of  $f$ on $X\setminus \Delta$ the set $P:=\cup_{l=0}^{n-1}f^{-l}(\Delta)$ is finite, and for $\rho:=d(f_1(c_1),P)>0$
the map $f^n$ is continuous on $(f_1(c_1)-\rho, f_1(c_1)+\rho)$.  Using the continuity of $f_1$ on $[c_0,c_1]$, we obtain that there exists $\delta'>0$ such that $|f^n(f_1(c_1))-f^{n+1}(x)|<\delta$ for all $x\in(c_1-\delta',c_1)$.  As $c_1\in\Delta_{lr}(x_0)$, there exists  $m$ such that $f^m(x_0)\in (c_1-\delta',c_1)$ and by the triangular inequality we deduce that  $|f^l(x_0)-z|<2\delta$ for $l=m+n+1$,  that is $f^l(x_0)\in(y,y+\epsilon)$.
\end{proof}

Before proving Assertion A, we show that together with Lemma \ref{DENSITY} it implies Theorem \ref{TH3}:

\begin{proof}[Proof of Theorem \ref{TH3} as a Corollary of Assertion A] Suppose  that Assertion A is true and let $f$ satisfying {\bf P1-6}\footnote{{\bf P1-6} is a shorthand notation for ``the properties {\bf P1} to  {\bf P6}".} for some $N\geq 2$.  Then, $f$ is a piecewise affine contracting interval map, it has the separation property, its attractor is a Cantor set and $\Delta=\Delta_{lr}(x)$ for any $x\in\widetilde{X}$. So, to prove Theorem \ref{TH3} it remains to prove that
$\Lambda$ is minimal and that  $\Delta=\Delta_{lr}(x)$ for any $x\in X\setminus\widetilde{X}$. To this end, note that {\bf P1-6} do not impose any condition on the definition of $f$ on $\Delta$, and therefore $f$ can be suitably defined  on this set. So, we can assume that $f(c_i)\in\{f_{i-1}(c_i),f_i(c_i)\}$ for any $i\in\{1,\dots,N-1\}$, which implies by  {\bf P4} that $f(\Delta)\subset\widetilde{X}$. Since $f$ satisfies {\bf P1-6}, it satisfies the hypotheses of the Lemma \ref{DENSITY}. It follows that  the orbit of any point $x_0\in\Lambda\cap\widetilde{X}$ is dense in $\Lambda$. Now, since $f( c)\in\widetilde{X}$ for all $c\in\Delta$, the orbit of a point  in $\Lambda\setminus\widetilde{X}$ is also dense in $\Lambda$.  We deduce  that $\Lambda$ is minimal. Finally,  {\bf P4}, {\bf P6} and $f(\Delta)\subset\widetilde{X}$ imply that for any $x\in X\setminus\widetilde{X}$ and $c\in\Delta$ there exists $p\geq 1$ such that  $c\in\Delta_{lr}(f^p(x))$.  Since $\Delta_{lr}(f^p(x))\subset \Delta_{lr}(x)$, we conclude that any discontinuity of $f$ is lr-recurrently visited by any orbit, ending the proof of Theorem \ref{TH3}. \end{proof}

In the following subsections  we will prove Assertion A. Let us describe briefly the route of the proof:

The proof goes by induction on the number $N \geq 2$ of contraction pieces of $f$.  In Subsection \ref{CIR1}, relying on a known example, we prove that there exists a  map  satisfying {\bf P1-6} for ${N}=2$. In Subsection \ref{CIRK}, we construct a map $g$ satisfying {\bf P1-6} with ${N+1}$ contraction pieces, assuming the existence of a map $f$ satisfying {\bf P1-6}  with $N$ contraction pieces  and with contracting  constant $\lambda \in (0,1)$. To construct $g$ from $f$, we first choose an adequate point $\xi_0 \in \widetilde X$, and its orbit $\{ \xi_r\}_{r \in\N}$, where $\xi_r = f^r(\xi_0)$. Second, we \lq\lq cut\rq\rq \ the interval $X$ at each point $\xi_r$ with $r \geq 1$ (but not at $\xi_0$), and insert an interval $G_r$ substituting the point $\xi_r$, such that, for all $r \geq 1$ the length of $G_r$ is $\lambda^r$. We   define an  affine map $g|_{G_r} : G_r \to G_{r+1}$ for all $r \geq 1$. In this way, we have added a new discontinuity point of $g$  at the point,  say $\xi_0$. For all $y \not \in \bigcup_{r \geq 1} G_r$ we define the image $g(y)$ from the image $f(x)$ of the corresponding  point $x \in X \setminus \{\xi_r\}_{r \geq 1}$. In particular $g$ preserves the old $N-1$ discontinuity points of $f$. So  $g$ has $N$ discontinuity points, hence $N+1$ continuity pieces. Finally, in Proposition \ref{PROPDELAMORT}  we show that there exists a good choice of the \em cutting orbit \em $\{\xi_r\}_{r \in \mathbb N}$, to make $g$ inherit   the properties {\bf P1-6} from $f$.

\subsection{Full complexity with a single  discontinuity}\label{CIR1}

In this subsection we prove that Assertion A holds for ${N}=2$. We begin with a lemma about the sets $\Lambda_n$ (defined in \eqref{ATTRACTOR}) for  piecewise increasing maps with two contraction pieces.

\begin{lemma}\label{CR_gaps} Let $X:=[0,1]$, $c\in(0,1)$ and $f:X\to X$ be a piecewise contracting map with contraction pieces $X_1=[0,c)$ and $X_2=(c,1]$. Suppose that the continuous extensions $f_1$ and $f_2$ of $f$
are increasing and such that $0=f_2(c )<f_2(1 )<f_1(0 )<f_1(c )=1$.

For every $k \in \mathbb{N}$, let $H_k := f^k((f(1),f(0)))$. If $c \notin \overline{H}_k$ for all $k \in \mathbb{N}$, then $H_k=(f^{k+1}(1),f^{k+1}(0))$ for all $k\in\N$ and
\begin{equation}\label{LAMBDAH}
 \Lambda_n = [0,1] \setminus \bigcup_{k=0}^{n-1} H_k \quad \forall n \geq 1.
\end{equation}
Moreover, $\{0,1\}\subset\widetilde{X}$, and for any $n\geq 1$ and $A\in\mathcal{A}_n$ there exists $p$ and $q$
in $\N$ such that $A=[f^p(0),f^q(1)]$.
\end{lemma}

\begin{proof}  Assume that $H_k=(f^{k+1}(1),f^{k+1}(0))$ for some $k\in\N$. Then $H_k\subset[0,c)$ or $H_k\subset(c,1]$, because $c\notin H_k$.
Hence, $f$ is continuous and increasing on $H_k$, and  $H_{k+1}=(f^{k+2}(1),f^{k+2}(0))$. As $H_0=(f(1),f(0))$, we have proved by induction that $H_k=(f^{k+1}(1),f^{k+1}(0))$ for every $k\in\N$.

Now let us show \eqref{LAMBDAH} by induction. We have $\Lambda_{1} := f_1( [0,c] ) \cup f_2( [c,1] ) = [f(0), 1] \cup [0, f(1)] = [0,1] \setminus H_0$ and hence \eqref{LAMBDAH} is true for $n=1$. Now let $n \geq 1$, and assume  that
$\Lambda_{n} = [0,1] \setminus \cup_{k=0}^{n-1} H_k$.
We shall prove that $\Lambda_{n+1} = [0,1] \setminus \bigcup_{k=0}^{n} H_k$.
First, observe that
\[
\Lambda_{n+1} = f_1\left(\overline{\Lambda_{n}\cap (0,c)}\right) \cup f_2\left(\overline{\Lambda_{n}\cap (c,1)}\right)= f (\Lambda_{n} \setminus \{c\}) \cup \{ f_1(c) , f_2(c) \},
\]
since $c\in [0,1]\setminus \cup_{k=0}^{n-1} \overline{H}_k$ and therefore it belongs to the interior of  $\Lambda_{n}$.
It follows that,
\begin{eqnarray*}
 \Lambda_{n+1}  = f \left( ([0,1] \setminus \{c\}) \setminus \cup_{k=0}^{n-1} H_k\right) \cup \{ 0 , 1 \}.
\end{eqnarray*}
Besides, since $f$  is injective on $[0,1]\setminus\{c\}$, we have
\[
f\left( ([0,1] \setminus \{c\}) \setminus \cup_{k=0}^{n-1} H_k\right) = f ([0,1] \setminus \{c\}) \setminus f ( \cup_{k=0}^{n-1} H_k).
\]
On the other hand, we have  $f(x)\notin\{0,1\}$ for any $x\neq c$ and $c\notin H_k$ for all $k\in \N$, which implies that
\[
\Lambda_{n+1} = \left( f ([0,1] \setminus \{c\})\cup \{ 0 , 1 \}\right)\setminus \cup_{k=0}^{n-1} f(H_k)
=([0,1]\setminus H_0)\setminus\cup_{k=1}^{n} H_k=[0,1]\setminus\cup_{k=0}^{n} H_k,
\]
as wanted.

Let us prove that $\{0,1\}\subset\widetilde{X}$. Since $c \notin f^k([f(1), f(0)])$ for every $k \geq 0$, we have $c\notin\{f^k(0), f^k(1)\}$ for all $k\geq 1$. It follows that $\{0,1\}\subset\widetilde{X}$, because $c\notin\{0,1\}$.

To end the proof, first observe  that $f_1([0,c])=[f(0),1]$ and $f_2([c,1])=[0,f(1)]$. Therefore, the atoms of $\mathcal{A}_1$ are of the form $[f^p(0),f^q(1)]$. Now, as an induction hypothesis,  assume that for some  $n\geq 1$ and for any $A\in\mathcal{A}_n$ there exists $p$ and $q$ in $\N$ such that $A=[f^p(0),f^q(1)]$. If $B\in\mathcal{A}_{n+1}$, then, by definition of atoms, there exist $i\in\{1,2\}$ and $A\in\mathcal{A}_n$  such that $B=f_i(\overline{A\cap X_i})$, where $X_1:=[0,c)$ and $X_2:=(c,1]$. If $c\notin A$, then $A\subset X_i$ and $B=f(A)=[f^{p+1}(0),f^{q+1}(1)]$. If $c\in A$, then either $A\cap X_i=[f^p(0),c)$ or $A\cap X_i=(c,f^q(1)]$. In both cases $A\cap X_i\neq\emptyset$, since $\{0,1\}\subset\widetilde{X}$. If follows that  either $B=[f^{p+1}(0),1]$ or $B=[0,f^{q+1}(1)]$.
\end{proof}

\begin{Proposition}\label{PRC1} Let $\lambda$ and $\mu\in(0,1)$ be such that $\lambda+\mu>1$ and denote $c:=(1-\mu)/\lambda$.
Let $f_1:[0,c]\to[0,1]$ and $f_2:[c,1]\to [0,1]$ be defined by
\begin{equation}\label{RC1}
f_{1}(x)=\lambda x+ \mu \quad\forall\,x\in[0,c]\quad\text{and}\quad f_{2}(x)=\lambda x+ \mu -1\quad\forall\,x\in[c,1].
\end{equation}
Then, any map $f:[0,1]\to[0,1]$ such that $f|_{[0,c)}=f_1|_{[0,c)}$ and $f|_{(c,1]}=f_2|_{(c,1]}$, and $c\notin f^k([f(1),f(0)])$ for all $k\in\N$, satisfies {\bf P1-6} with
$X_1=[0,c)$ and $X_2=(c,1]$.
\end{Proposition}

\begin{proof} Consider a map $f:[0,1]\to[0,1]$ which satisfies the hypotheses of Proposition \ref{PRC1}. Then, it is easy to check that $f$ satisfies {\bf P1-2}, with $c_0=0$, $c_1=c$ and $c_2 = 1$. Also, $f$ satisfies all the hypotheses of Lemma \ref{CR_gaps}, and it follows in particular that ${\bf P4}$ holds.

Let $\gamma\in\{0,1\}$ and let us show that $\{f^n(\gamma)\}_{n\in\N}$ belongs to $\Lambda$. Since for any $k\in\N$ the set $H_k$ of Lemma \ref{CR_gaps} is an open set of $[0,1]$, it cannot contain
$\gamma$.  Therefore, by \eqref{LAMBDAH} we have $\gamma\in\Lambda_n$ for all $n\in\N$, that is $\gamma\in\Lambda$.
Since $\gamma\in\widetilde{X}$ and $\Lambda\cap \widetilde{X}$ is forward invariant, we deduce that $\{f^n(\gamma)\}_{n\in\N}\subset\Lambda$.

 Now, let us prove that $\Lambda$ has no isolated point. Let $x \in \Lambda$ and $\epsilon >0$.
 Let $n \geq 1$ be such that $\text{diam}(A) < \epsilon$ for every $A \in \mathcal{A}_n$.
 Let $A \in \mathcal{A}_n$ be such that $x \in A$, and let $p$ and $q \in \mathbb{N}$ be such that $A = [f^p(0), f^q(1)]$. If $x \in ( f^p(0), f^q(1) ]$, then $0 < |f^p(0)-x| < \epsilon$ and if $x \in [ f^p(0), f^q(1) )$, then $0 < |f^q(1)-x| < \epsilon$. Since both points $f^p(0)$ and $f^q(1) \in \Lambda$, we found a point in $\Lambda\setminus\{x\}$ which is at a distance less than $\epsilon$ of $x$. This  shows that $\Lambda$ is a perfect set (recall that $\Lambda$ is compact) and it proves at the same time that $\{ f^n(0)\}_{n\in\N}$ and $\{ f^n(1)\}_{n\in\N}$ are dense in $\Lambda$, i.e. $f$ satisfies {\bf P5}. Now, $\Lambda$ is totally disconnected because $f$ satisfies the separation property \cite{CGMU16}. It follows that $\Lambda$ is a Cantor set and $f$ satisfies {\bf P3}.

Now we show that $f$ satisfies {\bf P6}. To this end, we prove that  for every $x \in \widetilde X$ and $\epsilon >0$ there exists $l\in\N$ and $r\in\N$ such that $f^{l}(x) \in (c-\epsilon,c)$ and $f^{r}(x) \in (c,c+\epsilon)$.
Let $x \in \widetilde X$ and $\epsilon >0$.
Let $n_0 \in \mathbb{N}$ be such that $\text{diam}(A) < \epsilon/2$ for every $A \in \mathcal{A}_{n_0}$ and let $A \in \mathcal{A}_{n_0}$ be such that $f^{n_0}(x) \in A$. Denote  $p$ and $q$ the integers such that $A = [f^p(0), f^q(1)]$ and let $T = \{ t \in \mathbb{N} : f^{n_0+t}(x) \in (c-\epsilon,c+\epsilon) \}$. Arguing by contradiction, assume that $T=\emptyset$ or that $f^{n_0+t}(x) \in (c,c+\epsilon)$ for all $t \in T$.
Then, by induction on $t\in\N$, we deduce that
\[
0<(f^{q+t}(1)-c)(f^{n_0+t}(x)-c) \quad\text{and}\quad  0\leq f^{q+t}(1)-f^{n_0+t}(x)<\epsilon/2\qquad\forall\,t\in\N.
\]
Therefore, for each $t\in T$ we have $f^{q+t}(1)\in(c,1]$ and for each $t\notin T$ we have
\[
|f^{q+t}(1)-c|\geq|f^{n_0+t}(x)-c|-|f^{q+t}(1)-f^{n_0+t}(x)|>\epsilon/2.
\]
We deduce that $f^{q+t}(1)\notin(c-\epsilon/2,c)$ for all $t\in\N$. Now, let $\nu>0$  be such that $f^k(1)\notin (c-\nu,c)$ for all $k\leq q$. Then, $f^k(1)\notin (c-\epsilon_0,c)$ for all $k\in\N$, where $\epsilon_0=\min\{\nu,\epsilon/2\}$. On the other hand, there exit $p'$ and $q'$ such that $[f^{p'}(0),f^{q'}(1)]$ is an atom of diameter strictly smaller than $\epsilon_0$ which contains $c$. Since $f^{p'}(0)\in\Lambda$ and $f^{k}(1) \notin (c-\epsilon_0,c)$ for all $k\in\N$,  it follows that  $\{ f^k(1) \}_{k \in \mathbb{N}}$ is not dense in $\Lambda$, which is a contradiction. Therefore,  $T\neq\emptyset$ and there exists $l\in\N$ such that $f^{l}(x)\in(c-\epsilon,c)$. Now, if we assume that $f^{n_0+t}(x) \in (c-\epsilon,c)$ for every $t \in T$, we deduce with an analogous proof that $\{ f^k(0) \}_{k \in \mathbb{N}}$ is not dense in $\Lambda$. Therefore,  there exists $r\in\N$ such that $f^{r}(x)\in(c,c+\epsilon)$.
\end{proof}

\begin{proof}[Proof of Assertion A for $N=2$]
Consider a map $f:[0,1]\to[0,1]$ defined by
\[
f(x)=\lambda x + \mu\mod 1\quad\forall x\in [0,1],
\]
where $\lambda$ and $\mu\in(0,1)$ are such that $\lambda+\mu>1$. Then, $f$ is a piecewise contracting map which satisfies the hypothesis \eqref{RC1}, with the particularity that $f(c)=0$.  Immediately, the ``gap"  between the atoms of generation 1 is the interval $(f(1), f(0))$. It is standard to prove  that if there exists a minimum  natural number $k $ such that $c\in f^k([f(1),f(0)])$, then the attractor of
$f$ contains only periodic points. On the other hand,  it has been proved  using a rotation number approach
that there is an uncountable set of values of $(\lambda,\mu)$ such that $f$ has no periodic points \cite{B93,BC99,C99,CFLM06,V87}.
It follows that, for such values of $(\lambda,\mu)$,   $c\notin f^k([f(1),f(0)])$ for all $k\in\N$.
Together with Proposition \ref{PRC1} this proves that Assertion A holds for $N=2$.
\end{proof}

\subsection{Full complexity with any  number of discontinuities}\label{CIRK}

In the previous subsection we  proved the existence of a map  satisfying {\bf P1-6} with  ${N}=2$ (a piecewise contracting interval map with a single
discontinuity). In this subsection we will complete the proof  of Assertion A, by induction on $N$.

 Let us assume that Assertion A holds for some $N\geq 2$. Then, there exists $c_0<c_{1}<\dots<c_{N}$ in $\R$
and $f:X\to X$, where  $X=[c_0,c_{N}]$, which satisfies {\bf P1-6} with the contraction pieces $X_1=[c_0,c_1),\dots,X_{N}=(c_{N-1},c_{N}]$. In the following, we denote $\Delta_f:=\{c_i\}_{1\leq i\leq {N-1}}$ the set of the discontinuities of $f$, $\Lambda_f$ the attractor of $f$, and $\widetilde{X}_f$  the set defined by  \eqref{XTILDE} where $\Delta=\Delta_f$.

Now, we   construct a new map $g$ with $N+1$ contraction pieces  and satisfying {\bf P1-6}, from the given map $f$ that satisfies {\bf P1-6} and has   $N$ contraction pieces.
The construction involves what we call a {\it well-cutting} orbit  of $f$,
 defined as follows:

\begin{definition}\em We say that an orbit $\{\xi_r\}_{r\in\N}$ of $f$ is \emph{well-cutting}, if $\xi_0\in\Lambda_f\cap\widetilde{X}_f$ and  $\{\xi_r\}_{r\in\N}$ does not contain any point of the following sets:

\noindent 1) the boundaries of the gaps of the Cantor set $\Lambda_f$,

\noindent 2) the orbits of $c_0$ and $c_{N}$,

\noindent 3) the orbits of $f_{i}(c_i)$ and $f_{i+1}(c_i)$ for all $i\in\{1,\dots,{N-1}\}$.

\end{definition}

Note that $f$ has an uncountable number of well-cutting orbits. Indeed, $\Delta_f$ and the sets of items  1), 2), 3) are countable. Therefore, the set $P$ of all their pre-images is also countable. Since $\Lambda_f$ is uncountable, the complement of $P$ in $\Lambda_f$ is uncountable and contains only  well-cutting orbit of $f$. Also, a well-cutting orbit is not eventually periodic. Indeed,  by Theorem \ref{THCOMPNBLR},  the property {\bf P6}  implies that the complexity function of the itinerary of $\xi_0$ is not eventually constant.

Let  $\lambda\in(0,1)$ be the slope of $f$ on any of its contraction pieces, and let $\{\xi_r\}_{r\in\N}$ be a well-cutting orbit of $f$. Consider the function $\phi:[c_0,c_{N}]\to\R$ defined for any $x\in[c_0,c_{N}]$ by
\begin{equation}\label{PHI}
\phi(x)=x+\sum\limits_{n\in\mathcal{N}(x)}\lambda^n\quad\text{where}\quad
\mathcal{N}(x):=\{n\geq 1 : \xi_n<x\}.
\end{equation}
The following lemma gathers  basic properties of $\phi$ that we will use in this section.

\begin{lemma}\label{LPHI}  The function $\phi$ is strictly increasing, left-continuous on $[c_0,c_N]$, continuous on $[c_0,c_N]\setminus\{\xi_r\}_{r\geq 1}$ (in particular at $\xi_0$), and $\lim\limits_{x\searrow\xi_r}\phi(x)=\phi(\xi_r)+\lambda^r$ for all $r\geq 1$. Moreover,
\[
\phi([c_0,c_N])=\left[\phi(c_0),\phi(c_N)\right]\setminus\bigcup_{r=1}^\infty G_r \quad\text{where}\quad G_r:=\left(\phi(\xi_r),\phi(\xi_r)+\lambda^r\right]\quad\forall\,r\geq 1,
\]
and $\overline{G}_r\cap \overline{G}_l=\emptyset$ for all $r\neq l$.
\end{lemma}

\begin{proof} Noting that $\mathcal{N}(x)\subset\mathcal{N}(x')$ for any $x<x'$, it is straightforward to show that $\phi$ is strictly increasing.

Now we show that the left-hand limit of $\phi$ at any point $x_0\in(c_0,c_N]$ is equal to $\phi(x_0)$. Let $x_0\in(c_0,c_N]$ and $\epsilon>0$. Take $m_0$ such that
$\sum_{n=m_0}^\infty\lambda^{n}<\epsilon/2$ and let $\rho>0$ be such that $\xi_n\notin(x_0-\rho,x_0)$ for all $n<m_0$. Now, if $\delta:=\min\{\epsilon/2,\rho\}$, then for any $x\in(x_0-\delta,x_0)$, we have
\[
|\phi(x_0)-\phi(x)|= x_0-x+\sum_{n\in\mathcal{N}(x_0)}\lambda^n -\sum_{n\in\mathcal{N}(x)}\lambda^n<\frac{\epsilon}{2}+\sum_{n\in\mathcal{N}(x_0)\setminus\mathcal{N}(x)}\lambda^n<\epsilon,
\]
since $\min{\mathcal{N}(x_0)\setminus\mathcal{N}(x)}\geq m_0$ if  $\xi_n\notin(x_0-\rho,x_0)$ for all $n<m_0$.
With an analog proof, we can show that the right-hand limit of $\phi$ at $x_0\in[c_0,c_N)$ is equal to $\phi(x_0)$ if $x_0\neq\xi_r$ for any $r\geq 1$, and equal to $\phi(\xi_r)+\lambda^r$ if $x_0=\xi_r$ for some $r\geq 1$.

Taking into account that $\phi$ is strictly increasing, left-continuous and has a discontinuity jump
of magnitude $\lambda^r$ at every point $\xi_r$ with $r\geq 1$, it is standard to check that
\[
\phi([c_0,c_N])=\left[\phi(c_0),\phi(c_N)\right]\setminus\bigcup_{r=1}^\infty G_r.
\]

Finally we show that $\overline{G}_r\cap \overline{G}_l=\emptyset$ for all $r\neq l$. Let $l$ and $r\geq 1$ be such that $r\neq l$ and $\xi_r>\xi_l$. As $\phi(\xi_r)\in \phi([c_0,c_N])$ we have that $\phi(\xi_r)\notin {G}_l$, and as $\phi$ is injective we have that $\phi(\xi_r)\notin \overline{G}_l$. It follows that $\overline{G}_r\cap \overline{G}_l=\emptyset$, since $\phi(\xi_r)>\phi(\xi_l)$.
\end{proof}

The following proposition shows how a well-cutting orbit of $f$ and its associated function $\phi$ allow us to obtain a map  which satisfies {\bf P1-6} with ${N+1}$ contraction pieces.
Therefore, it ends the proof  of Assertion A by induction on $N$.

\begin{Proposition}\label{PROPDELAMORT}  Let $\{\xi_r\}_{r\in\N}$ be a well-cutting orbit of $f$ and $\phi$ be defined according to \eqref{PHI}. Let $\Delta_g:=\phi(\Delta_f\cup\{\xi_0\})$. Then, any map $g:[\phi(c_0),\phi(c_{N})]\to[\phi(c_0),\phi(c_{N})]$ defined on $[\phi(c_0),\phi(c_{N})]\setminus\Delta_g$ by
\begin{equation}\label{DEFG}
g(y)=\left\{\begin{array}{lcl}
\phi\circ f\circ\phi^{-1}(y)&\text{if}& y\in\phi([c_0,c_{N}])\setminus\Delta_g\\
\lambda(y-\phi(\xi_r))+\phi(\xi_{r+1})&\text{if}& y\in G_r\text{ and } r\geq 1
\end{array}
\right.
\end{equation}
satisfies {\bf P1-6} with $N+1$ contraction pieces.
\end{Proposition}

\begin{proof}  Let $d_0<d_1<\dots<d_{N+1}$ be such that $\{d_0,d_1,\dots,d_{N+1}\}=\phi(\Delta_{f}\cup\{c_0,\xi_0,c_{N}\})$. We denote $j_0$ the integer of
$\{1,\dots,N\}$ such that
\begin{equation}\label{DJ0}
d_{j_0}=\phi(\xi_0).
\end{equation}
Let $Y_1:=[d_0,d_1),Y_2:=(d_1,d_2),\dots,Y_{N+1}:=(d_N,d_{N+1}]$ and $Y:=[d_0,d_{N+1}]$.
Since $\phi$ is strictly increasing, the sets $Y_j$ are all non-empty and pairwise disjoint.  Let $g:Y\to Y$ be a map satisfying \eqref{DEFG}. We are going to show that $g$ satisfies {\bf P1-6} with $Y_1,Y_2,\dots,Y_{N+1}$.

\medskip
\noindent {\bf P1)} We first show that  for any  $j\in\{0,1,\dots,N\}$ the map $g$ is affine with slope $\lambda$ on the interval $(d_j,d_{j+1})$, that is
\begin{equation}\label{AFFINE}
g(y')-g(y)=\lambda(y'-y)\qquad\forall\, y,y'\in (d_j,d_{j+1}).
\end{equation}
To prove \eqref{AFFINE}, we fix $j\in\{0,1,\dots,N\}$ and $y,y'\in(d_j,d_{j+1})$, and we consider three cases:

\noindent Case 1: Assume that  $y,y'\in\phi([c_0,c_{N}])$ and denote $x:=\phi^{-1}(y)>\phi^{-1}(d_j)$ and $x':=\phi^{-1}(y')<\phi^{-1}(d_{j+1})$.
If we assume (with no loss of generality) $y<y'$, then $x<x'$ and
\[
g(y')-g(y)=f(x')-f(x)+\sum_{n\in\mathcal{N}_2}\lambda^n \qquad\text{where}\qquad
\mathcal{N}_2:=\{n\geq 1: f(x)\leq \xi_n<f(x')\}.
\]
Since $f$ is affine on $(\phi^{-1}(d_j),\phi^{-1}(d_{j+1}))$ and has  slope $\lambda$,
we have that $f(x')-f(x)=\lambda(x'-x)$. On the other hand, since $f$ is injective (separation property)
and increasing on $(\phi^{-1}(d_j),\phi^{-1}(d_{j+1}))$ we have that $n\in\mathcal{N}_2$ if and only if $n\geq 2$ and $n-1\in
\mathcal{N}_1$, where $\mathcal{N}_1:=\{n\geq 1: x\leq \xi_n<x'\}$. It follows that
\[
g(y')-g(y)=\lambda(x'-x)+\sum_{n\in\mathcal{N}_1}\lambda^{n+1}=\lambda(\phi(x')-\phi(x))=\lambda(y'-y).
\]

\noindent  Case 2: Assume $y\in G_r$ for some $r\geq 1$ and $y'\in\phi([c_{0},c_{N}])$. Since $G_r$ is an interval and $G_r\cap\phi([c_{0},c_{N}])=\emptyset$, we have $G_r\subset (d_j,d_{j+1})$. Therefore, $\phi(\xi_r)\in(d_j,d_{j+1})$, because $r\neq 0$ and $\{\xi_r\}_{r\in\N}$ is a well-cutting orbit. So we can use Case 1 to obtain
\[
g(y') - g(\phi(\xi_r))=\lambda(y'-\phi(\xi_r)).
\]
On the other hand, by definition of $g$ on $G_r$
\[
g(\phi(\xi_r))-g(y)=\phi(\xi_{r+1})-\lambda(y-\phi(\xi_r))-\phi(\xi_{r+1}),
\]
and the sum of these two equalities is $\eqref{AFFINE}$.

\noindent  Case 3: Assume $y\in G_r$ and $y'\in G_r'$ for some $r$ and $r'\geq 1$ and let $z\in\phi([c_0,c_N])\cap(d_j,d_{j+1})$. Then applying Case 2 twice, we obtain
\[
g(y')-g(y)=g(y')-g(z)+g(z)-g(y)=\lambda(y'-y),
\]
which ends the proof of \eqref{AFFINE}.

 According to \eqref{AFFINE}, we know that $g$ is continuous on $\cup_{j=0}^{N}(d_j,d_{j+1})$. To study the continuity of $g$ on the border of these intervals, we compute the left-hand and right-hand limits of $g$
at the points of $\phi(\Delta_f\cup\{c_0,c_{N}\})$. Let $i\in\{1,2,\dots,N\}$ and consider an increasing sequence $\{x_n\}_{n\in\N}$ in $(c_{i-1},c_{i})$ which converges to $c_i$. Then, the sequence $\{\phi(x_n)\}_{n\in\N}$ is increasing and converges to $\phi(c_i)$, because $\phi$ is continuous
at any point that  is not in the positive orbit of $\xi_0$, and $\{\xi_r\}_{r\in\N}$ is a well-cutting orbit.
Using the definition of $g$ on $\phi([c_0,c_{N}])\setminus\Delta_g$ and the continuous extension $f_i$ of $f|_{(c_{i-1},c_{i})}$ to $[c_{i-1},c_{i}]$ we obtain that
\[
\lim_{n\to\infty}g(\phi(x_n))=\lim_{n\to\infty}\phi(f(x_n))=\lim_{n\to\infty}\phi(f_i(x_n))=\lim_{x\to f_i(c_i)}\phi(x).
\]
Once again, since $\{\xi_r\}_{r\in\N}$ is a well-cutting orbit, $f_i(c_i)$ does not belong to $\{\xi_r\}_{r\in\N}$ and $\phi$ is continuous at $f_i(c_i)$. It follows that
\begin{equation}\label{LLIMIT}
\lim_{y\nearrow \phi(c_i)}g(y)=\phi(f_i(c_i)).
\end{equation}
Similarly,  for any $i\in\{0,1,\dots,{N-1}\}$  we have
\begin{equation}\label{RLIMIT}
\lim_{y\searrow \phi(c_i)}g(y)=\phi(f_{i+1}(c_i)).
\end{equation}
Since $f_{i+1}(c_i)\neq f_{i}(c_i)$ for any $i\in\{1,\dots,{N-1}\}$ (separation property) and $\phi$ is injective, we deduce that $g$ is discontinuous on $\phi(\Delta_f)$. On the other hand, from \eqref{LLIMIT} and \eqref{RLIMIT} respectively, we obtain that $g$ is continuous at $\phi(c_{N})$ and $\phi(c_{0})$ respectively.

It remains to study $g$ at $\phi(\xi_0)$. As $\{\xi_r\}_{r\in\N}$ is a well-cutting orbit, there exists $i\in\{1,2,\dots,{N}\}$
such that $\xi_0\in(c_{i-1},c_{i})$. Using the left continuity of $\phi$,
the monotonicity of $f$ on $(c_{i-1},c_{i})$, and the continuity of $f$ at $\xi_0$, we obtain that
\begin{equation}\label{LLIMITD}
\lim_{y\nearrow \phi(\xi_0)}g(y)=\phi(f(\xi_0))=\phi(\xi_1).
\end{equation}
On the other hand
\begin{equation}\label{RLIMITD}
\lim_{y\searrow \phi(\xi_0)}g(y)=\lim_{x\searrow \xi_1}\phi(x)=\phi(\xi_1)+\lambda.
\end{equation}
Equalities \eqref{LLIMITD} and \eqref{RLIMITD} prove that $g$ is discontinuous at $\phi(\xi_0)$. We conclude that the set of the discontinuity points of $g$ is $\Delta_g$, which together
with \eqref{AFFINE}  proves that $g$ is a piecewise affine contracting map with contractions pieces $Y_1,Y_2,\dots,Y_{N+1}$. In particular, the set of \eqref{XTILDE} writes for $g$ as
$\widetilde{X}_g=\bigcap _{n= 0}^{+ \infty} g^{-n}(X\setminus\Delta_g)$.

\medskip
\noindent{\bf P2)} Let us prove  that $g$ satisfies the separation property. For any $j\in\{1,\dots,N+1\}$ denote $g_j:\overline{Y_j}\to[\phi(c_0),\phi(c_{N})]$ the continuous extension of $g|_{Y_j}$ to $\overline{Y_j}$. Then,
\[
\overline{Y_{j_0}}=[\phi(c_{j_0-1}),\phi(\xi_0)],\ \overline{Y_{j_0+1}}=[\phi(\xi_0),\phi(c_{j_0})]
\quad\text{and}\quad
\overline{Y_j}=\left\{\begin{array}{lcl}
[\phi(c_{j-1}),\phi(c_{j})] &\text{if}& 1\leq j <j_0\\

[\phi(c_{j-2}),\phi(c_{j-1})] &\text{if}& j_0+1< j \leq N+1
\end{array}
\right.,
\]
where $j_0$ is defined by \eqref{DJ0}.
For every $j\in\{1,\dots,N+1\}$ the map $g_j$ is affine with slope $\lambda> 0$. Therefore,  $g_j(\overline{Y}_j)$ is an interval whose boundaries are obtained using \eqref{LLIMIT}, \eqref{RLIMIT}, \eqref{LLIMITD} and \eqref{RLIMITD}. Indeed,
\[
g_{j_0}(\overline{Y_{j_0}})=[\phi(f_{j_0}(c_{j_0-1})),\phi(f_{j_0}(\xi_0))],\quad g_{j_0+1}(\overline{Y_{j_0+1}})=[\phi(f_{j_0}(\xi_0))+\lambda,\phi(f_{j_0}(c_{j_0}))],
\]
and
\[
g_j(\overline{Y_j})=\left\{\begin{array}{lcl}
[\phi(f_{j}(c_{j-1})),\phi(f_{j}(c_{j}))] &\text{if}& 1\leq j <j_0\\

[\phi(f_{j-1}(c_{j-2})),\phi(f_{j-1}(c_{j-1}))] &\text{if}& j_0+1< j \leq N+1
\end{array}
\right. .
\]
On the one hand we have obtained that  $g_{j_0}(\overline{Y}_{j_0})\cap g_{j_0+1}(\overline{Y}_{j_0+1})=\emptyset$.
On the other hand, the monotonicity of $\phi$ and the separation property of $f$ imply that the sets
\[
g_1(\overline{Y_1}),\dots, g_{j_0-1}(\overline{Y_{j_0-1}}), g_{j_0}(\overline{Y_{j_0}})\cup g_{j_0+1}(\overline{Y_{j_0+1}}), g_{j_0+2}(\overline{Y_{j_0+2}}),\dots, g_{N+1}(\overline{Y_{N+1}})
\]
are pairwise disjoint. It follows that $g$ satisfies the separation property.

\medskip
\noindent{\bf P3)} Now we study the attractor of $g$, which can be written using the continuous extensions of $g$ as
\[
\Lambda_g=\bigcap_{n=1}^{\infty}\Lambda_{g,n},
\]
where the sets $\Lambda_{g,n}$ are recursively defined as
\[
\Lambda_{g,1}=\bigcup_{j=1}^{N+1}g_j(\overline{Y_j})\quad\text{and}\quad
\Lambda_{g,n+1}=\bigcup_{j=1}^{N+1}g_j(\overline{\Lambda_{g,n}\cap Y_j})\quad\forall\, n\geq 1.
\]
We denote $G:=\bigcup_{r=1}^\infty G_r$ and we recall that $\phi(X)=Y\setminus G$ by Lemma \ref{LPHI}.

Let us show by induction that
\begin{equation}\label{LGNPHILFN}
\Lambda_{g,n}\cap\phi(X)=\phi(\Lambda_{f,n})\qquad\forall\,n\geq 1.
\end{equation}
First note that
\[
\Lambda_{g,1}\cap\phi(X)=\bigcup_{j=1}^{N+1}(g_j(\overline{Y_j}\cap\phi(X))\cup g_j(\overline{Y_j}\cap G))\cap\phi(X)=\bigcup_{j=1}^{N+1}g_j(\overline{Y_j}\cap\phi(X))\cap\phi(X),
\]
since for any $j\in\{1,\dots,N+1\}$ we have $\overline{Y_j}\cap G=Y_j\cap G$ and therefore $g_j(\overline{Y_j}\cap G)=g(Y_j\cap G)\subset G$. Besides, if $y\in\overline{Y_j}\cap\phi(X)$ and $y\neq d_{j_0}$ then
\begin{equation}\label{GJFJ}
g_j(y)=\left\{\begin{array}{lcl}
\phi\circ f_j\circ\phi^{-1}(y)&\text{if}& 1\leq j < j_0+1\\

\phi\circ f_{j-1}\circ\phi^{-1}(y)&\text{if}& j_0+1\leq j \leq N+1
\end{array}
\right. ,
\end{equation}
and if $y\neq d_{j_0}$ then
\begin{equation}\label{GJ0FJ0}
g_{j_0}(y)=\phi\circ f_{j_0}\circ\phi^{-1}(y)\quad\text{and}\quad g_{j_0+1}(y)=\phi(f_{j_0}(\xi_0))+\lambda=\phi(\xi_1)+\lambda\in G.
\end{equation}
Denote $Z_1,Z_2,\dots,Z_{N}$ the sets defined by
\[
{Z}_{j_0}={Y}_{j_0}\cup{Y}_{j_0+1}
\quad\text{and}\quad
{Z}_j=\left\{\begin{array}{lcl}
{Y}_{j} &\text{if}& 1\leq j <j_0\\

{Y}_{j+1}&\text{if}& j_0< j \leq N
\end{array}
\right. ,
\]
then
\[
\Lambda_{g,1}\cap\phi(X)=\bigcup_{j=1}^{N}\phi\circ f_j\circ\phi^{-1}(\overline{Z_j}\cap\phi(X)).
\]
As $\overline{Z_j}\cap\phi(X)=\phi(\overline{X_j})$ for any $j\in\{1,\dots,N\}$, we deduce that
\[
\Lambda_{g,1}\cap\phi(X)=\bigcup_{j=1}^{N}\phi(f_j(\overline{X_j}))=\phi(\Lambda_{f,1}).
\]
Now assume that $\Lambda_{g,n}\cap\phi(X)=\phi(\Lambda_{f,n})$ for some $n\geq 1$. As before, we have
\[
\Lambda_{g,n+1}\cap\phi(X)=\bigcup_{j=1}^{N+1}g_j(\overline{\Lambda_{g,n}\cap Y_j}\cap\phi(X))\cap\phi(X)=
\bigcup_{j=1}^{N}\phi\circ f_j\circ\phi^{-1}(\overline{\Lambda_{g,n}\cap Z_j}\cap\phi(X)).
\]
To obtain that $\Lambda_{g,n+1}\cap\phi(X)=\phi(\Lambda_{f,n+1})$ and complete the induction, it is enough to show that
\begin{equation}\label{FUCKINGFERMETURE}
\overline{\Lambda_{g,n}\cap Z_j}\cap\phi(X)=\phi(\overline{\Lambda_{f,n}\cap X_j}).
\end{equation}
On the one hand, using the induction hypothesis, we obtain  $\Lambda_{g,n}\cap\phi(X)=\phi(\Lambda_{f,n})$. On the other hand we have  ${Z_j}\cap\phi(X)=\phi({X_j})$. As $\phi$ is injective, it follows that
\[
\Lambda_{g,n}\cap Z_j\cap\phi(X)=\phi(\Lambda_{f,n}\cap{X_j})\quad\forall\, j\in\{1,\dots,N\}.
\]
Let $y\in\phi(\overline{\Lambda_{f,n}\cap X_j}\setminus(\Lambda_{f,n}\cap X_j))$ and  $x\in\overline{\Lambda_{f,n}\cap X_j}\setminus(\Lambda_{f,n}\cap X_j)$ be such that $y=\phi(x)$. Then, $x\in\{c_{j-1},c_{j}\}$ and $\phi$ is continuous at $x$, since $x\notin\{\xi_r\}_{r\in\N}$. Let $\{x_n\}_{n\in\N}$ be a sequence of $\Lambda_{f,n}\cap X_j$ which converges to $x$. Then, the sequence $\{\phi(x_n)\}_{n\in\N}$ belongs to $\Lambda_{g,n}\cap Z_j$ and converges to $y=\phi(x)$. It follows that $y\in\overline{\Lambda_{g,n}\cap Z_j}\cap\phi(X)$ and  we have proved that $\phi(\overline{\Lambda_{f,n}\cap X_j})\subset\overline{\Lambda_{g,n}\cap Z_j}\cap\phi(X)$. To show the converse inclusion, take $y\in\overline{\Lambda_{g,n}\cap Z_j}\cap\phi(X)\setminus(\Lambda_{g,n}\cap Z_j)$. Then,  $y\in\overline{Z_j}\setminus Z_{j}\subset\{\phi(c_{j-1}),\phi(c_{j}),\phi(\xi_0)\}$. If $y=\phi(\xi_0)$, then $j=j_0$ and $y\in\phi(\Lambda_{f,n}\cap{X_{j_0}})$, since $\xi_0\in\Lambda_f\cap X_{j_0}$. As $c_{j-1}$ and $c_{j}$ are lr-recurrently visited, there exist two atoms $A_{j-1}$ and $A_j$ in the set of the atoms of generation $n$ of $f$ such that $c_{j-1}$ and $c_{j}$ belong to the interior of $A_{j-1}$ and $A_j$, respectively. As $\{c_{j-1},c_j\}\subset\overline{X_j}$, it follows that $c_{j-1}\in\overline{A_{j-1}\cap X_j}$ and  $c_{j}\in\overline{A_{j}\cap X_j}$. Recalling that $\Lambda_{f,n}$ is the union of all the atoms of generation $n$ of $f$, we deduce that $\{c_{j-1},c_j\}\subset\overline{\Lambda_{f,n}\cap X_j}$ and $y\in\phi(\overline{\Lambda_{f,n}\cap X_j})$. This ends the proof of \eqref{FUCKINGFERMETURE} and completes the proof by induction of \eqref{LGNPHILFN}.

Using the injectivity of $\phi$ and \eqref{LGNPHILFN} we obtain that
\begin{equation}\label{LAMBDAINTERPHI}
\Lambda_g\cap\phi(X)=\phi(\Lambda_f).
\end{equation}
Now we show that $\Lambda_g\cap G=\{\phi(\xi_r)+\lambda^r\}_{r\geq 1}$. To this end, we prove by induction that
\begin{equation}\label{LAMBDANINTERG}
\Lambda_{g,n}\cap G=\{\phi(\xi_r)+\lambda^r,1\leq r\leq n\}\cup\bigcup_{r= n+1}^\infty G_r\quad\forall\, n\geq 1.
\end{equation}

Using equations \eqref{GJFJ} and \eqref{GJ0FJ0} it follows that
\[
\Lambda_{g,1}\cap G=\bigcup_{j=1}^{N+1}(g_j(\overline{Y_j}\cap \phi(X))\cup g_j(\overline{Y_j}\cap G))\cap G=
\{g_{j_0+1}(d_{j_0})\}\cup\bigcup_{j=1}^{N+1}g_j(\overline{Y_j}\cap G)\cap G
\]
Recalling the inclusion of $G$ in the union of the sets $Y_1,Y_2,\dots,Y_{N+1}$, and the definition of $g$ in $G$, we obtain
\[
\Lambda_{g,1}\cap G=\{\phi(\xi_1)+\lambda\}\cup\bigcup_{r=2}^{N+1}G_r,
\]
which proves \eqref{LAMBDANINTERG} for $n=1$.

Now,  assume that \eqref{LAMBDANINTERG} holds for some $n\geq 1$.  As $d_{j_0}=\phi(\xi_0)\in\phi(\Lambda_f)=\Lambda_g\cap\phi(X)$ we have that $d_{j_0}\in\Lambda_{g}\subset \Lambda_{g,n}$. On the other hand, there exits a decreasing sequence in $\Lambda_f$ which converges to $\xi_0$, because $\Lambda_f$ is a Cantor set and $\xi_0$ is not a border of a gap of $\Lambda_f$. By \eqref{LAMBDAINTERPHI}, the image of this sequence by $\phi$
belongs to $\Lambda_g$ and it is decreasing by monotonicity of $\phi$. From the continuity of $\phi$ at $\xi_0$, it follows that  $d_{j_0}\in\overline{\Lambda_{g}\cap Y_{j_0+1}}\subset\overline{\Lambda_{g,n}\cap Y_{j_0+1}}$. Now, using once again equations \eqref{GJFJ} and \eqref{GJ0FJ0}, we deduce
\[
\Lambda_{g,n+1}\cap G=\bigcup_{j=1}^{N+1}(g_j(\overline{\Lambda_{g,n}\cap Y_j}\cap\phi(X))\cup g_j(\overline{\Lambda_{g,n}\cap Y_j}\cap G))\cap G=
\{g_{j_0+1}(d_{j_0})\}\cup\bigcup_{j=1}^{N+1}g_j(\overline{\Lambda_{g,n}\cap Y_j}\cap G)\cap G.
\]
As $\overline{\Lambda_{g,n}\cap Y_j}\setminus(\Lambda_{g,n}\cap Y_j)\subset\{d_{j-1},d_{j}\}$, and $d_0,d_1,\dots,d_{N+1}$ do not belong to $G$, we have that $g_j(\overline{\Lambda_{g,n}\cap Y_j}\cap G)=g(\Lambda_{g,n}\cap Y_j\cap G)$ for all $j\in\{1,2,\dots,N+1\}$. Therefore,
\[
\Lambda_{g,n+1}\cap G=\{\phi(\xi_1)+\lambda\}\cup g(\Lambda_{g,n}\cap G).
\]
Using the induction hypothesis, we obtain
\[
\Lambda_{g,n+1}\cap G=\{\phi(\xi_r)+\lambda^r,1\leq r\leq n+1\}\cup\bigcup_{r= n+2}^\infty G_r,
\]
which complete the proof of \eqref{LAMBDANINTERG}.

Note that \eqref{LAMBDANINTERG} can also be written as
\[
\Lambda_{g,n}\cap G=\{\phi(\xi_r)+\lambda^r\}_{r\geq 1}\cup\bigcup_{r= n+1}^\infty \text{Int}(G_r)\quad\forall\, n\geq 1,
\]
and recall that the sets $G_r$ are pairwise disjoint. This implies that $\Lambda_{g}\cap G=\{\phi(\xi_r)+\lambda^r\}_{r\geq 1}$,  which together with \eqref{LAMBDAINTERPHI} gives
\begin{equation}\label{LAMBDAG}
\Lambda_g=\phi(\Lambda_f)\cup\{\phi(\xi_r)+\lambda^r\}_{r\geq 1}.
\end{equation}

Now we show that $\Lambda_g$ is  a Cantor set. By definition of attractor $\Lambda_g$ is compact. Besides, it is totally disconnected because $g$ satisfies the separation property (see Theorem 5.2 of \cite{CGMU16}). It remains to show that  $\Lambda_g$ has no isolated point. Let $y\in\phi(\Lambda_f)$ and $x\in\Lambda_f$ be such that $y=\phi(x)$. As  $\Lambda_f$ is a Cantor set,
there exists a sequence $\{x_n\}_{n\in\N}$ in  $\Lambda_f\setminus\{x\}$ which converges to $x$ and $\{\phi(x_n)\}_{n\in\N}$ belongs to $\Lambda_g\setminus\{y\}$.  If $y\notin\{\phi(\xi_r)\}_{r\geq 1}$, then $\phi$ is continuous at $x$  and $\{\phi(x_n)\}_{n\in\N}$ converges to $y$. If $y=\phi(\xi_r)$ for some $r\geq 1$, then we can assume that $\{x_n\}_{n\in\N}$
is increasing, since $\xi_r$ does not belong to the boundaries of the gaps of the Cantor set $\Lambda_f$.  Using the left-continuity of $\phi$, we obtain that $\{\phi(x_n)\}_{n\in\N}$ converges to $y$. Now, let $y=\phi(\xi_r)+\lambda^r$ for some $r\geq 1$ and let $\{x_n\}_{n\in\N}$ be a decreasing sequence in $\Lambda_f$ which converges to $\xi_r$. Then,  $\{\phi(x_n)\}_{n\in\N}$ converges to $y$ and belongs to $\Lambda_g\setminus\{y\}$, since $y\notin\phi(X)$.

\medskip
\noindent{\bf P4)} Let us prove that $g_{j}(d_j)$ and $g_{j+1}(d_j)$ belong to $\widetilde{X}_g$ for every $j\in\{1,2,\dots,N\}$. First note that for any $x$ such that
$f^n(x)\notin\Delta_f\cup\{\xi_0\}$ for all $n\in\N$, we have
\begin{equation}\label{CONJUG}
g^n\circ\phi(x)=\phi\circ f^n(x)\quad\forall n\in\N \quad\text{and}\quad \phi(x)\in\widetilde{X}_g.
\end{equation}
Let $j\in\{1,2,\dots,N\}$ and $y\in\{g_{j}(d_j),g_{j+1}(d_j)\}$. If
$y\neq g_{j_0+1}(d_{j_0})$, then using \eqref{GJFJ} and \eqref{GJ0FJ0} we obtain that $y=\phi(x)$ for some point
$x\in\{f_j(c_j),f_{j+1}(c_j),f_{j-1}(c_{j-1}),f_{j}(c_{j-1}),f(\xi_0)\}$. Since $\{\xi_n\}_{n\in\N}$ is a well-cutting orbit and $f_{i}(c_i)$ and $f_{i+1}(c_i)$ belong to $\widetilde{X}_f$ for any $i\in\{1,2,\dots,{N-1}\}$, it follows from \eqref{CONJUG} that $y\in\widetilde{X}_g$. Now, if $y=g_{j_0+1}(d_{j_0})$ then by \eqref{GJ0FJ0}
we have that $g^n(y)\in G$ for all $n\in\N$, and therefore $y\in\widetilde{X}_g$.

\medskip
\noindent{\bf P5)} Let $i\in\{1,\dots,{N-1}\}$ and $l\in\{i,i+1\}$ be such that $\{f^n(f_l(c_i))\}_{n\in\N}$ is dense in $\Lambda_f$.
Let us denote $x_0:=f_l(c_i)$ and $y_0:=\phi(x_0)$. Using \eqref{GJFJ} we obtain that there exists $j\neq j_0$ such that $y_0\in\{g_j(d_j),g_{j+1}(d_j)\}$. Therefore, to prove {\bf P5} for $g$ it is enough to show that  $\{g^n(y_0)\}_{n\in\N}$ is dense in $\Lambda_g$. Applying \eqref{CONJUG} to $x_0$, we obtain that $y_0\in\widetilde{X}_g$ and that $g^n(y_0)=\phi(f^n(x_0))$ for all $n\in\N$. Besides, according to \eqref{LAMBDAINTERPHI} we have $\{\phi(f^n(x_0))\}_{n\in\N}\subset\Lambda_g$, since $\{f^n(x_0)\}_{n\in\N}\subset\Lambda_f$.

First let $y\in\Lambda_g\setminus\{\phi(\xi_r),\phi(\xi_r)+\lambda^r\}_{r\geq 1}$. Then,
$y\in\phi(\Lambda_f)$ and there exists $\{n_k\}_{k\in\N}$ such that $\{f^{n_k}(x_0)\}_{k\in\N}$ converges to $x:=\phi^{-1}(y)\in\Lambda_f$. Since $x\notin\{\xi_r\}_{r\geq 1}$, it follows that $\phi$ is continuous at $x$ and $\{g^{n_k}(y_0)\}_{k\in\N}$ converges to $y$.

Now, let $y=\phi(\xi_r)$ for some
$r\geq 1$. Since $c_i$ is lr-recurrently visited by $\{f^k(x_0)\}_{k\in\N}$, and $\Lambda_f\cap (\xi_r-\nu,\xi_r)\neq\emptyset$ for all $\nu>0$ (recall that $\xi_r$ is not border of gap) by Lemma \ref{DENSITY} we have that the orbit of $x_0$ accumulates from the left
on $\xi_r$. Using the left-continuity of $\phi$ we obtain that the orbit of $y_0$ accumulates
on $y=\phi(\xi_r)$. Using now $\Lambda_f\cap (\xi_r,\xi_{r}+\nu)\neq\emptyset$ for all $\nu>0$, we obtain that
there exists a subsequence of $\{f^{n}(x_0)\}_{n\in\N}$ which converges to $\xi_r$ from the right-hand side. The image  by $\phi$ of this subsequence converges  to $\phi(\xi_r)+\lambda^r$. It follows that $\phi(\xi_r)+\lambda^r$ is also a limit point of
$\{g^n(y_0)\}_{n\in\N}$.

\medskip
\noindent{\bf P6)}  Now we prove that all the discontinuities of $g$ are lr-recurrently visited by the orbits of the points of
$\widetilde{X}_g$. Let $j\in\{1,\dots,N\}$ and $y\in\widetilde{X}_g$. We are going to show that there exist
two sequences $\{n_k\}_{k\in\N}$ and $\{m_k\}_{k\in\N}$ going to infinity such that $g^{n_k}(y)<d_j<g^{m_k}(y)$ for all $k\in\N$ and $\{g^{n_k}(y)\}_{k\in\N}$ and $\{g^{m_k}(y)\}_{k\in\N}$ converge to $d_j$.

First, let us show the above assertion for $j\neq j_0$. We denote $c$ the point of $\Delta_f$ such that $\phi( c)=d_j$. If $y\in\widetilde{X}_g\cap\phi(X)$, then
\begin{equation}\label{CONJUGACY}
g^n(y)=\phi\circ f^n\circ\phi^{-1}(y)\quad \forall\,n\in\N,
\end{equation}
since $\widetilde{X}_g\cap\phi(X)$ is forward invariant by $g$. It follows that $x:=\phi^{-1}(y)$ belongs to $\widetilde{X}_f$ and $c$ is lr-recurrently visited by $\{f^k(x)\}_{k\in\N}$  by property {\bf P6} of $f$. Therefore, there exist two sequences $\{n_k\}_{k\in\N}$ and $\{m_k\}_{k\in\N}$ going to infinity such that $f^{n_k}(x)<c<f^{m_k}(x)$ for all $k\in\N$ and $\{f^{n_k}(x)\}_{k\in\N}$ and $\{f^{m_k}(x)\}_{k\in\N}$ converge to $c$. Using \eqref{CONJUGACY}, the monotonicity of $\phi$ and its continuity at $c$,
we deduce that $\{g^{n_k}(y)\}_{k\in\N}$ and $\{g^{m_k}(y)\}_{k\in\N}$ satisfy the required properties.

Now assume  $y\in G\subset\widetilde{X_g}$ and let $r\geq 1$ be such that $y\in G_r$. Then, $g^n(y)\in G_{r+n}$ for all $n\in\N$, that is
\begin{equation}\label{YDANSG}
\phi(f^n(\xi_r))<g^n(y)\leq \phi(f^n(\xi_r))+\lambda^{n+r}\quad\forall\,n\in\N.
\end{equation}
As $\xi_r\in\widetilde{X}_f$, we have that $c$ is lr-recurrently visited by $\{f^k(\xi_r)\}_{k\in\N}$. Therefore, there exists two sequences $\{n_k\}_{k\in\N}$ and $\{m_k\}_{k\in\N}$ going to infinity such that $f^{n_k}(\xi_r)<c<f^{m_k}(\xi_r)$ for all $k\in\N$ and $\{f^{n_k}(\xi_r)\}_{k\in\N}$ and $\{f^{m_k}(\xi_r)\}_{k\in\N}$ converge to $c$. Using \eqref{YDANSG} and the continuity of $\phi$ at $c$ we have that both $\{g^{n_k}(y)\}_{k\in\N}$ and $\{g^{m_k}(y)\}_{k\in\N}$
converge to $\phi( c)=d_j$. On the other hand, by monotonicity of $\phi$ and the left-hand side of \eqref{YDANSG}
we have that $d_j<g^{m_k}(y)$ for all $k\in\N$. Again by monotonicity of $\phi$, we have $\phi(f^{n_k}(\xi_r))<d_j$, which implies that $\phi(f^{n_k}(\xi_r))+\lambda^{n_k+r}<d_j$ because $d_j\notin G$. We deduce from  \eqref{YDANSG} that $g^{n_k}(y)<d_j$.

Finally, we show that $d_{j_0}=\phi(\xi_0)$ is lr-recurrently visited by the orbit of any point of $\widetilde{X}_g$. Let $d_j\in\Delta_g\setminus\{d_{j_0}\}$
be a discontinuity of $g$ such that $\{g^n(g_l(d_j))\}_{n_\in\N}$ is dense in $\Lambda_g$ for some $l\in\{j,j+1\}$ (we proved in {\bf P6}  that it exists). Since $d_j\neq d_{j_0}$ we already know that $d_j$ is lr-recurrently visited by $\{g^k(y)\}_{k\in\N}$ for any $y\in\widetilde{X}_g$. Now, since $\phi(\xi_0)$ is not a border of gap of $\Lambda_g$ (because $\xi_0$ is not a border of gap of $\Lambda_f$ and $\phi$ is continuous at this point) we can apply Lemma \ref{DENSITY} to deduce that $d_{j_0}$ is lr-recurrently visited by $\{g^k(y)\}_{k\in\N}$ for any $y\in\widetilde{X}_g$.
\end{proof}

\begin{lemma} \label{N01}If any atom of generation $1$ of $f$ contains at most one point of $\Delta_f$, then there exists a well-cutting orbit of $f$ such that any atom of generation $1$ of $g$ contains at most one point of $\Delta_g$.
\end{lemma}

\begin{proof} As $f$ satisfies {\bf P6}, any discontinuity of $f$ is contained in an atom of generation $1$ of $f$
and this atom is unique because of {\bf P2}. It follows that one of the $N$ atoms of generation $1$ of $f$ does not contain any discontinuity. Let us denote $A_{i_1}$ this atom, where  $i_1\in\{1,\dots,N\}$ is such that $A_{i_1}=\overline{f((c_{i_1-1},c_{i_1}))}$. Then, $A_{i_1}\cap\Lambda_{f,n}\neq\emptyset$ for all $n\geq 1$, since by {\bf P6} for any $n\geq 2$ there exists an atom of generation $n-1$ which contains $c_{i_1}$ in its interior. It follows that
$A_{i_1}\cap\Lambda_{f}\neq\emptyset$. Moreover, $A_{i_1}\cap\Lambda_{f}$ is compact, totally disconnected and any
point of $\Int(A_{i_1})\cap\Lambda_{f}$ is not isolated, because $\Lambda_f$ is a Cantor set. Now, as the atoms are compact and disjoint, if $x\in (A_{i_1}\setminus\Int(A_{i_1}))\cap\Lambda_{f}$ then $x$ is a border of gap of $\Lambda_f$ and there exists a sequence in $\Int(A_{i_1})\cap\Lambda_{f}$ which converges to $x$. We deduce that
$A_{i_1}\cap\Lambda_{f}$ is a Cantor set. Therefore, there exists $\xi_0\in A_{i_1}$ such that $\{f^{k}(\xi_{0})\}_{k\in\N}$ is a well-cutting orbit of $f$.

For every $i\in\{1,\dots,N\}$ let  $B_{i}=\overline{g((\phi(c_{i-1}),\phi(c_{i})))}$, where $\phi$ is the function defined in \eqref{PHI} with a well-cutting orbit of $f$ such that $\xi_0\in A_{i_1}$. Then, for any atom $B$ of generation 1 of $g$ there exists $i\in\{1,\dots,N\}$ such that $B\subseteq B_i$. Moreover, we can show that
\[
B_i\cap\phi(X)=\phi(A_i)\qquad\forall\, i\in\{1,\dots,N\},
\]
where $A_i:=\overline{f((c_{i-1},c_{i}))}$ is an atom of generation 1 of $f$. Now, let $d\neq d'\in\Delta_g$ and assume by contradiction that there exists an atom of generation 1 of $g$ which contains $d$ and $d'$. Then, there exits an atom of generation 1 of $f$ which contains two elements of
$\Delta_f\cup\{\xi_0\}$, which is a contradiction.
\end{proof}

\begin{proof}[Proof of part 2) of Theorem \ref{TH}] From Theorem \ref{TH3} we deduce that equality \eqref{COMP2TH} of Theorem \ref{TH} holds for any $n\geq n_0$, where $n_0$ is defined in Lemma \ref{PROCOMPNBNLR} and is not necessarily equal to 1. For any piecewise contracting map $f$, the integer
$n_0$ is bounded above by $n_{0,f}:=\min\{n\geq 1: \# A\cap\Delta_f\leq1\ \forall\,A\in\mathcal{A}_{f,n}\}$, where $\mathcal{A}_{f,n}$ is the set of the atoms of generation $n$ of $f$.
Obviously, if $f$ has only two contraction pieces  then $n_{0,f}=1$. Lemma \ref{N01} proves that if $f$ satisfies
{\bf P1-6} and $n_{0,f}\neq1$, then for a suitable choice of the well-cutting orbit of $f$,  we have $n_{0,g}=1$ for the map $g$ of Proposition \ref{PROPDELAMORT}.
\end{proof}

\vspace{5ex}

\noindent {\bf \emph{Acknowledgements.}}
EC and PG have been supported by the Math-Amsud Regional Program 16-MATH-06 (PHYSECO).
EC thanks the invitations of Instituto Venezolano de Investigaciones Cient\'{\i}ficas (IVIC-Venezuela) and Universidad de Valpara\'{\i}so (CIMFAV Chile), and the partial financial support of CSIC (Universidad de la Rep\'{u}blica, Uruguay) and ANII (Uruguay). AM and PG thank the invitation of  Universidad de la Rep\'{u}blica and the partial financial support of CSIC (Uruguay). AM thanks the financial support of the CIMFAV CID 04/03, and the members of the CIMFAV for their kindness
during his stay.

\bibliographystyle{amsplain}

\begin{thebibliography}{99}
%-------------------------------------------------------------------------
\bibitem{AB98}
P. Alessandri and V. Berth\'e,
Three distance theorems and combinatorics on words,
Enseignement Math\'{e}matique {\bf 44} (1998) 103--132.
%-------------------------------------------------------------------------
\bibitem{B06}
J. Br\'emont,
Dynamics of injective quasi-contractions,
Ergodic Theory and Dynamical Systems {\bf 26} (2006) 19--44.
%-------------------------------------------------------------------------
\bibitem{B93}
Y. Bugeaud,
Dynamique de certaines applications contractantes lin\'eaires par morceaux sur [0, 1[,
Comptes Rendus Acad. Sci. Paris, Ser. I {\bf 317} (1993) 575--578.
%-------------------------------------------------------------------------
\bibitem{BC99}
Y. Bugeaud and  J.-P. Conze,
Calcul de la dynamique de transformations lin\'eaires contractantes mod 1 et arbre de Farey,
Acta Arithmetica {\bf 88} (1999) 201--218.
%-------------------------------------------------------------------------
\bibitem{CGMU16}
E. Catsigeras, P. Guiraud, A. Meyroneinc and E. Ugalde,
On the asymptotic properties of piecewise contracting maps,
Dynamical Systems {\bf 67} (2016) 609--655.
%-------------------------------------------------------------------------

 \bibitem{C99}
R. Coutinho,
Din\^amica simb\'olica linear,
Ph.D. Thesis, Technical University of Lisbon (1999).
%-------------------------------------------------------------------------
 \bibitem{CFLM06}
R. Coutinho, B. Fernandez, R. Lima and A. Meyroneinc,
Discrete time piecewise affine models of genetic regulatory networks,
Journal of Mathematical Biology {\bf 52} (2006) 524--570.
%-------------------------------------------------------------------------
\bibitem{D99}
G. Didier,
Caract\'erisation des N-\'ecritures et application \`a l'\'etude des suites de complexit\'e ultimement $n + cste$,
Theoretical Computer Science {\bf 215}  (1999) 31--49.

\bibitem{GT88}
J.-M. Gambaudo  and Ch. Tresser,
On the dynamics of quasi-contractions,
Bulletin of the Brazilian Mathematical Society {\bf 19} (1988) 61--114.
%-------------------------------------------------------------------------
\bibitem{K75}
M.S. Keane,
Interval exchange transformations,
Math. Zeitsch. {\bf 141} (1975) 25--31.
%-------------------------------------------------------------------------
\bibitem{KR06}
B. Kruglikov and M. Rypdal,
A piecewise affine contracting map with positive entropy,
Discrete and Continuous Dynamical Systems {\bf 16} (2006) 393--394.

%-------------------------------------------------------------------------
\bibitem{LU06}
R. Lima and E. Ugalde,
Dynamical complexity of discrete-time regulatory networks,
Nonlinearity {\bf 19} (2006) 237--259.
%-------------------------------------------------------------------------
\bibitem{N1}
A. Nogueira, B. Pires and R. Rosales,
Asymptotically periodic piecewise contractions of the interval,
Nonlinearity {\bf 27} (2014) 1603--1610.
%-------------------------------------------------------------------------
\bibitem{N2}
A. Nogueira and B. Pires,
Dynamics of piecewise contractions of the interval,
Ergodic Theory and Dynamical Systems {\bf 35} (2015) 2198--2215.
%-------------------------------------------------------------------------
\bibitem{N3}
A. Nogueira, B. Pires and R. Rosales,
Topological dynamics of piecewise $\lambda$-affine maps,
Ergodic Theory and Dynamical Systems, DOI: https://doi.org/10.1017/etds.2016.104.
%-------------------------------------------------------------------------
\bibitem{P16}
B. Pires,
Invariant measures for piecewise continuous maps,
Comptes Rendus Acad. Sci. Paris, Ser. I {\bf 354} (2016) 717--722.
%-------------------------------------------------------------------------

\bibitem{V87}
P. Veerman,
Symbolic dynamics of order-preserving orbits,
Physica 29D (1987) 191-201.
%-----------------------------------------------------------------------
\bibitem{F13}
 S. Ferenczi,
Combinatorial methods for interval exchange transformations,
Southeast Asian Bulletin of Mathematics {\bf 37} (2013) 47--66.
%-------------------------------------------------------------------------
\bibitem{FZ08}
S. Ferenczi and C. Zamboni,
Languages of $k$-interval exchange transformations,
Bulletin London Mathematical Society {\bf 40} (2008) 705--714.
%-------------------------------------------------------------------------
\end{thebibliography}

\end{document}